\documentclass[a4paper,12pt]{article}

\usepackage{amsmath,amssymb,amstext,amsfonts,latexsym,amsthm}
\usepackage{bbold}
\usepackage{mathptmx}
\usepackage{dsfont}
\usepackage{graphicx}
\usepackage{graphicx,hyperref}
\usepackage[T1]{fontenc}
\usepackage[utf8]{inputenc}
\usepackage{authblk}

\usepackage[shortlabels]{enumitem}

\newtheorem{theorem}{Theorem}

\newtheorem{proposition}{Proposition}[section]
\newtheorem{definition}{Definition}[section]
\newtheorem{corollary}{Corollary}[theorem]

\newtheorem{example}{Example}[section]
\newcounter{Th-Alfa}

\newcommand{\ZZp}{\mathds{Z}_{+}}
\newcommand{\NN}{\mathds{N}}
\newcommand{\PP}{\mathds{P}}
\newcommand{\RR}{\mathds{R}}
\newcommand{\dsty}{\displaystyle}
\newcommand{\Jm}{\mu^{{\scriptscriptstyle \!\alpha\!,\beta }}\!}
\newcommand{\fJm}[2]{\mu^{\scriptscriptstyle \! \! #1,#2}\!}
\newcommand{\Jw}{w_{\scriptscriptstyle \!\! \alpha,\beta}\!}
\newcommand{\owj}[1][j]{\overline{w}^{ k}_{#1}}
\newcommand{\twj}[1][j]{\widetilde{w}^{k}_{#1}}
\newcommand{\fowj}[2][j]{\overline{w}^{#2}_{#1}}
\newcommand{\ow}{\overline{w}}
\newcommand{\tw}{\widetilde{w}}
\newcommand{\Lp}[1][p]{\mathit{L}^{#1}}
\newcommand{\JLp}[1][p]{\mathit{L}^{#1}\!(\Jm)}
\newcommand{\WWp}[1][p]{\mathbf{W}^{\scriptscriptstyle \alpha,\beta}_{\scriptscriptstyle \varpi,#1}}
\newcommand{\UUp}[1][p]{\mathbf{U}^{\scriptscriptstyle \alpha,\beta}_{\scriptscriptstyle \varpi,#1}}
\newcommand{\uup}[1][p]{{U}^{\alpha,\beta}_{\varpi,#1}}
\newcommand{\Bp}[1][p]{\mathit{B}_{#1}}
\newcommand{\normJLp}[2][p]{\| #2 \|_{\scriptscriptstyle \alpha,\beta,#1}}
\newcommand{\sob}{{\scriptscriptstyle\mathsf{s}}}
\newcommand{\normSp}[2][p]{\| #2 \|_{\sob,{\scriptscriptstyle#1}}}
\newcommand{\normLp}[3][p]{\| #3 \|_{\Lp[#1]\!\!\left(#2\right)}}
\newcommand{\IpS}[2]{\langle #1,#2 \rangle_{\!\sob}}
\newcommand{\IpJ}[2]{\langle #1,#2 \rangle_{\scriptscriptstyle \! \! \alpha,\beta}}
\newcommand{\IPJ}[2]{\left\langle #1,#2 \right\rangle_{\! \scriptscriptstyle \!\!\!\alpha,\beta}}
\newcommand{\Pn}[1][n]{P_{#1}^{\scriptscriptstyle (\alpha,\beta)}\!}
\newcommand{\pn}[1][n]{p_{#1}^{\scriptscriptstyle (\alpha,\beta)}\!}
\newcommand{\fPn}[3][n]{P_{#1}^{\scriptscriptstyle (#2,#3)}\!}
\newcommand{\fpn}[3][n]{p_{#1}^{\scriptscriptstyle (#2,#3)}\!}
\newcommand{\sn}[1][n]{q_{#1}\!}
\newcommand{\PRn}[1][n]{P_{\varpi,#1}^{\scriptscriptstyle(\alpha,\beta)}\!}
\newcommand{\hn}[1][n]{h_{#1}^{\scriptscriptstyle \left( \alpha ,\beta \right)}}

\newcommand{\funD}[4]{\begin{cases}#1 , & \hbox{if } \; #2,\\[2pt]
									#3, & \hbox{if } \; #4 \end{cases}}
									
\newcommand{\funT}[6]{\begin{cases} #1 , & \hbox{if } \; #2,\\
						#3, & \hbox{if } \; #4,\\
						#5, & \hbox{if } \; #6
					  \end{cases}}						
\newcommand{\funC}[8]{\begin{cases} #1 , & \hbox{if } \; #2,\\
						#3, & \hbox{if } \; #4,\\
						#5, & \hbox{if } \; #6,\\
						#7, & \hbox{if } \; #8
					  \end{cases}}

\newcommand{\FSJn}[1][n]{S_{#1}^{\scriptscriptstyle \alpha,\beta}\!}
\newcommand{\FSSn}[1][n]{S_{#1}}

\title{Discrete-Continuous Jacobi-Sobolev Spaces  \\ and  Fourier Series}
\author[1]{Abel D\'{\i}az-Gonz\'{a}lez\thanks{abdiazgo@math.uc3m.es}}
\author[1]{Francisco Marcell\'{a}n-Espa\~{n}ol\thanks{pacomarc@ing.uc3m.es}}
\author[1]{H\'{e}ctor Pijeira-Cabrera\thanks{hpijeira@math.uc3m.es}}
\author[2]{Wilfredo Urbina-Romero\thanks{wurbinaromero@roosevelt.edu}}
\affil[1]{Universidad Carlos III de Madrid, Spain}
\affil[2]{Roosevelt University, USA}

\date{}

\begin{document}
\maketitle

\begin{abstract}Let  $p\geq 1$, $\ell\in \NN$, $\alpha,\beta>-1$ and $\varpi=(\omega_0,\omega_1, \dots, \omega_{\ell-1})\in \RR^{\ell}$. Given a suitable function $f$, we define the  discrete-continuous Jacobi-Sobolev norm of $f$ as:
$$
	\normSp{f}:= \left(\sum_{k=0}^{\ell-1} \left|f^{(k)}(\omega_{k})\right|^{p} + \int_{-1}^{1} \left|f^{(\ell)}(x)\right|^{p} d\Jm(x)\right)^{\frac{1}{p}},
$$
where $ d\Jm(x)=(1-x)^{\alpha} (1+x)^{\beta}dx$.   Obviously,  $\normSp[2]{\cdot}= \sqrt{\IpS{\cdot}{\cdot}}$, where $\IpS{\cdot}{\cdot}$ is the inner product.
$$
   \IpS{f}{g}:= \sum_{k=0}^{\ell-1} f^{(k)}(\omega_{k}) \, g^{(k)}(\omega_{k}) + \int_{-1}^{1} f^{(\ell)}(x) \,g^{(\ell)}(x) d\Jm(x).
$$
In this paper, we summarize  the main advances on the convergence of the Fourier-Sobolev series, in norms of type $L^p$, cases continuous and discrete.  We study the completeness of the Sobolev space of functions associated with the norm $\normSp{\cdot}$ and the denseness of the polynomials. Furthermore,  we obtain the conditions for the convergence in  $\normSp{\cdot}$ norm of the partial sum of the Fourier-Sobolev series of   orthogonal polynomials with respect to $\IpS{\cdot}{\cdot}$ .
\end{abstract}


\section{Introduction}

The convergence problem of the trigonometric Fourier series has been one of the main driving forces behind the mathematical analysis growth. For instance, the development of the  Hilbert spaces theory  is closely related to the problem of the convergence of the Fourier series for square integrable $2\pi$-periodic functions on $[-\pi,\pi],$
$$ \lim\limits_{n\rightarrow \infty }\int_{-\pi}^\pi \left| f\left( x\right) -\sum\limits_{|k| \leq n} \left( \frac{1}{2\pi} \int_{-\pi}^{\pi} f(y) e^{-iky} dy\right) e^{ikx} \right| ^{2} dx =0.$$

In 1927, M. Riesz gave an analytic proof of the convergence of trigonometric Fourier series for $\mathit{L}^p$-integrable $2 \pi$-periodic functions on $[-\pi,\pi]$ using complex methods (see \cite{rie}). This proof was one of the sources to develop the main harmonic analysis tools   such as the  Hardy-Littlewood maximal functions and  the  singular integrals.  In the proof,  Riesz strongly used the conjugated function, which is the first singular integral considered. Later A. P. Calder\'{o}n and A. Zygmund studied singular integrals in higher dimensions using real analysis methods (see \cite{CalZyg}).

 Trough the Hilbert space theory the study of the $\mathit{L}^2$ convergence of the classical orthogonal polynomials Fourier series became a straightforward issue,  since it only depends  on the orthogonality. But we find a very different scenario in the case of the $\mathit{L}^p$ convergence when $p\neq2$.  The study of $\mathit{L}^p$-convergence of classical orthogonal polynomials Fourier series started in the 1940's   with a serie of papers (see \cite{Poll46,Poll47,Poll48,Poll49}). H. Pollard considered the mean convergence of the Legendre, Gegenbauer and Jacobi Fourier expansions extending the harmonic analysis techniques used for Fourier series to  this case (see Theorem \ref{ThPollard}). Pollard used a special decomposition of the Christoffel-Darboux kernel such that the $\mathit{L}^p$ continuity of the Hilbert transform with respect to certain weights can be used.  Later on  B. Muckenhoupt revisited the convergence of Jacobi polynomials (see \cite{Muck69} and Theorem \ref{ThMuck}).

Additionally,   H. Pollard proved in \cite{Poll48} by means of the asymptotic relation between the Hermite polynomials and the Gegenbauer polynomials (see \cite[(5.3.4)]{Szg75}) that the Fourier expansions in terms of the Hermite polynomials only converge in $\mathit{L}^p$-norm when $p=2$. Finally, due to the relation between the Hermite and the Laguerre polynomials, (see \cite[(5.6.1)]{Szg75}), there  exists also an anomalous behavior on the $\mathit{L}^p$-convergence of the Laguerre polynomials Fourier series.

Let $\mathds{F}$ be a linear space of complex-valued functions with an inner product $\langle \cdot,\cdot \rangle$. The inner product $\langle \cdot,\cdot \rangle$ is said to be standard if $\langle zf(z),g(z) \rangle=\langle f(z),z g(z) \rangle$ for every $f,g \in \mathds{F}$, i.e.  standard means that the operator of multiplication by the independent variable $z$ is symmetric.

 We denote by  $\PP$  the  linear space of all polynomials. If $\PP$ is a subspace of $\mathds{F}$, well-known arguments allow to establish the existence of a unique (up to normalization) sequence of orthogonal polynomials with respect to $\langle \cdot,\cdot \rangle$. If this inner product is standard, as an immediate consequence we get that the corresponding sequence of orthogonal polynomials satisfies a  three-term recurrence relation. This is a very important property, connected with several fields   such as: difference equations (\cite{MaMePi13}), numerical analysis (\cite{Gaut04}), Fourier analysis (\cite{Jack41,Osil99,Will19}), among others.  When you do not have such a relation, these approaches are not generally available.

 Otherwise, if an inner product is not symmetric with respect to the multiplication operator, we say that it is non-standard. This is the case of inner products modified by terms involving derivatives. Here we restrict ourselves to the  Sobolev-type inner products, defined as:
 \begin{equation}\label{SobolevProd}
\IpS{f}{g} =\sum_{k=0}^{\ell}   \int_{-\infty}^\infty f^{(k)}(x) g^{(k)}(x) d\mu_k(x), \qquad \text{where } \; \ell \in \ZZp \; \text{ is fixed}.
\end{equation}
and the associated Sobolev norms of type-$p$:
 \begin{equation}\label{Sobolev-pNorm}
\normSp{f}= \Big( \sum_{k=0}^{\ell}   \int_{-\infty}^\infty  \left|f^{(k)}(x) \right|^p  d\mu_k(x)\Big)^{1/p} \qquad \text{where } \; 1\leq  p < \infty \; \text{ is fixed}.
\end{equation}

Observe that the measures $\mu_k$ can be either a discrete measures  or a continuous one. Thus there are essentially three types of such inner products:

\begin{enumerate}
\item[I.]  \emph{Discrete case.} The  support  $\mu_0$ contains infinitely many points  and $\mu_1, \cdots, \mu_{\ell}$ are supported on finite subsets.
\item[II.]   \emph{Continuous case.}  All the measures involved in the inner product \eqref{SobolevProd} (i.e. the norm \eqref{Sobolev-pNorm}) are supported  on subsets with infinitely many points.
\item[III.]   \emph{Discrete-continuous case.} The  support  $\mu_{\ell}$  contains infinitely many points and $\mu_0, \cdots, \mu_{\ell-1}$ are supported on finite subsets.
\end{enumerate}

 The theory of Sobolev ortogonal polynomials  has  been studied intensively in the last 30 years,  we recommend the survey \cite{MarcellanXu} and the references therein to get a good overview.  As it would expect, the orthogonal projection onto the usual $\Lp[2]$ space does not have the required properties, which leads naturally to consider inner products of the form \eqref{SobolevProd}.
 Nonetheless, there are only a few general results about  the extremal polynomials with respect to the norm \eqref{Sobolev-pNorm}, with $p\neq 2$. Most of them, on the  asymptotic behavior of the sequence of the corresponding extremal polynomials (see \cite{AbPiLo19,LopMarPij06,LoPePi05}).

 From  the anomalous behavior of the Fourier expansions in terms of the Hermite and Laguerre polynomials with respect to the $\mathit{L}^p$-convergence, it is fruitless to consider   the convergence of the Fourier series of any of them in the Sobolev case. Therefore,  due to the importance of the Jacobi polynomials in this topic, in the next section, we review   some of their properties and also the main ideas of the proof of the $\mathit{L}^p$ convergence of their Fourier series.   In the section \ref{Sect-SobolevPoly}, we review what is known about the convergence of the  Fourier series of Jacobi-Sobolev polynomials.

Section \ref{Sect-CarSobSpac} is dedicated to defining the discrete-continuous Jacobi-Sobolev spaces and to study necessary and sufficient conditions for its completeness.  In the last section we prove  the convergence of the Fourier series in the norm of the discrete-continuous Jacobi-Sobolev spaces.

 \section{Fourier series of Jacobi polynomials}

The Jacobi polynomials with parameters $\alpha,\beta >-1,$ $\{\Pn\}_{n\geq 0}$ are orthogonal polynomials
with respect to the inner product
\begin{equation}\label{JacobiIP}
   \IpJ{f}{g}  =  \int_{-1}^{1} f(x) \,g(x) \,d\Jm(x),
\end{equation}
where $ 	d\Jm(x)= \Jw(x)dx=(1-x)^{\alpha} (1+x)^{\beta}dx$, is the {\em Jacobi measure}  (or beta measure). We take the normalization of $\Pn$ such that
	$$ \Pn(1)=\binom{n+\alpha}{n}, \quad \text{where} \quad  \binom{a}{b}={\Gamma(a+1)}/\left({\Gamma(a-b+1)\Gamma(b+1)}\right), $$
for $a,b\in\RR$ and $\Gamma$ denotes the usual Gamma function (see \cite[\S 10.1]{Will19}). More precisely,
	$$\int^{\infty}_{-\infty} \Pn(x) \Pn[m](x) d\Jm(x) =\hn\delta_{n,m},\quad \text{for $n,m =0,1,2, \dots,$}$$ 		
and
\begin{equation}\label{l2norm}
	\hn=\frac{ 2^{\alpha +\beta+1}}{(2n+\alpha +\beta +1)}\frac{\Gamma \left( n+\alpha +1\right)\Gamma \left(n+\beta +1\right) }{\Gamma \left( n+1\right) \Gamma\left( n+\alpha +\beta +1\right) } \qquad \cite[(4.3.3)]{Szg75}.
\end{equation}
Additionally,
\begin{equation} \label{Jacobi-Der}
	\frac{d\Pn}{dx}\left( x\right)=\frac{\left( n+\alpha +\beta +1\right) }{2} \, \fPn[n-1]{\alpha +1}{\beta +1}\left(x\right) \qquad						\cite[(4.21.7)]{Szg75}.
\end{equation}
Also, $\left\{\Pn \right\}_{n\geq0} $ satisfy the three term recurrence relation,
\begin{align*}
			2n(n+\alpha +\beta)\Pn \left( x\right) = & (2n+\alpha +\beta -1) \left(\left( 2n+\alpha +\beta \right)x+\alpha ^{2}-\beta ^{2}\right)\,\Pn[n-1]\left( x\right)\\
			&  -2\left( n+\alpha -1\right) \left(n+\beta -1\right)\frac{ 2n+\alpha +\beta }{2n + \alpha +\beta -2} \Pn[n-2]\left( x\right),
\end{align*}
for $n\geq 2$;  with $\Pn[0](x)=1,$ and $\Pn[1](x)=\frac{1}{2} \left(\alpha +\beta +2\right) x+\frac{1}{2} \left( \alpha -\beta \right).$

For $1\leq p \leq \infty$ let consider the Banach space $\Lp(\Jm)$ of $p$-th power integrable functions with respect to $d\Jm$  on $[-1,1]$, embedded with the norm
\begin{equation}\label{Jacobi-pNorm}
  \normJLp{f}=\funD{\dsty \left(\int_{-1}^1|f(x)|^pd\Jm(x)\right)^{\frac{1}{p}}}{p\in [1,\infty)}{\dsty \inf\{a>0:\nu\left(\{|f|>a\}\right)=0\}}{p=\infty;}
\end{equation}
where $\nu$ denotes the Lebesgue measure on $[-1,1]$.

So,  $\{\Pn\}_{n\geq 0} $ is a sequence of orthogonal polynomials in $\Lp[2](\Jm)$. Let
\begin{equation}\label{OrthNormal-Jacobi}
 \pn=(\hn)^{-1/2}\Pn
\end{equation}
be  the $n$-th Jacobi orthonormal polynomial with respect to the inner product \eqref{JacobiIP}, then  $\normJLp[2]{\pn}=1$, where $\normJLp[2]{\cdot}=\sqrt{\IpJ{\cdot}{\cdot}}$. From \eqref{l2norm}-\eqref{Jacobi-Der}, we have the following property for the $k$-th derivative of the orthonormal Jacobi polynomials:
\begin{align*}
  \left(\pn(x)\right)^{(k)}=   &  A_{n.k}\; \fpn[n-k]{\alpha+k}{\beta+k}(x), \quad \text{where }\\ \nonumber
A_{n,k}= &   \sqrt{  \frac{\Gamma(n+1)\; \Gamma(n+\alpha+\beta+k+1) }{\Gamma(n-k+1)\; \Gamma(n+\alpha+\beta+1)}}\neq0,
\end{align*}
and  $f^{(k)}$ denotes the $k$-th derivative of a function $f$.

Given $f \in  \JLp[1]$,  set
\begin{equation}\label{FourCoef}
	a_n:=\IPJ{f}{\pn}=\int_{-1}^1 f(x) \pn(x) d\Jm(x),
\end{equation}
the $n$-th Fourier coefficient and define the partial sum operators of $f$ as

\begin{equation}\label{FourPartialSum}
	\FSJn(f,x)=\sum\limits_{k=0}^{n} a_{k}\pn[k](x).
\end{equation}

The analysis of the convergence of Fourier-Jacobi series in $\Lp(\Jm)$ has a long history. The first result was obtained by H. Pollard in \cite{Poll48} for Gegenbauer polynomials and then in  \cite{Poll49} for Jacobi polynomials,

\begingroup
\setcounter{Th-Alfa}{\value{theorem}}
\setcounter{theorem}{0} 
\renewcommand\thetheorem{\Alph{theorem}}
\begin{theorem}{\cite[Th. A]{Poll49}}\label{ThPollard}
Let $ \alpha,\beta \geq -1/2,$
		$$ M(\alpha, \beta) = 4 \; \max \left\{\frac{\alpha +1}{2 \alpha +3}, \frac{\beta +1}{2 \beta +3} \right\}\; \text{ and } \; 				m(\alpha, \beta) = 4 \; \min \left\{\frac{\alpha +1}{2 \alpha +1},\frac{\beta +1}{2 \beta +1} \right\}.$$
	Then, for any  $f\in\Lp(\Jm)$ with $M(\alpha, \beta) < p < m(\alpha, \beta)$, the Fourier-Jacobi expansion of $f$ converges to  $f$ in $			\Lp(\Jm)$, i.e.
	\begin{equation}\label{conv}
		\lim\limits_{n\rightarrow \infty } \normJLp{f-\FSJn(f,\cdot)}= \lim\limits_{n\rightarrow \infty }\left( \int_{-1}^{1}\left|f\left( x\right)- \FSJn(f,x) \right| ^{p}\,d\Jm(x)\right) ^{\frac{1}{p}}=0.
	\end{equation}
\end{theorem}
\endgroup

Furthermore, in \cite[Th. B]{Poll49}, Pollard also proved that the preceding result fails if $p<M(\alpha, \beta) $ or $p>m(\alpha, \beta) $.  In \cite{NewmanRudin}, J. Newman and W. Rudin, study the case of the extreme points  $p=M(\alpha, \beta) $ or $p=m(\alpha, \beta) $ proving that in that case the  convergence also fails.

According to Banach-Steinhaus uniform boundedness principle (see \cite[\S 5.8]{Rud87}), in order to  prove \eqref{conv}, it is enough to prove the uniform $\Lp(\Jm)$-boundedness of the partial sums, i.e. there exists a constant $C>0$  such that
$$ \normJLp{\FSJn(f,\cdot)} \leq  C \;\normJLp{f} ,$$
 for all $f \in\Lp(\Jm)$ and for all $n\geq 0.$

B. Muckenhoupt extended Pollard's result for all  $\alpha, \beta > -1$ (instead of   $\alpha, \beta \geq -1/2$) in the Corollary of Theorem 1 of  \cite{Muck69}. Muckenhoupt's proof uses  Pollard's decomposition but now for the weights  $w_{ap,bp} (x) = (1-x)^{ap} (1+x)^{bp}$, where  $a, b \in\RR$, such that
\begin{equation}\label{Muck-Cond}
	\left|a\!+\!\frac{1}{p}\!-\!\frac{\alpha +1}{2}\right|\! \leq\! \min\!\left\{\frac{1}{4}, \!\frac{\alpha +1}{2}\!\right\} \quad \!\text{and} \quad \left|b\!+\!\frac{1}{p}\!-\!\frac{\beta+1}{2}\!\right|\!\leq\!\min\!\left\{\! \frac{1}{4},\!\frac{\beta +1}{2}\right\}\!.
\end{equation}

Note that, if  $a=\alpha/p$ and $b =\beta/p$, we get  the Jacobi weights. It is easy to see that these conditions, for the case $\alpha, \beta \geq -1/2,$ give Pollard's conditions $M(\alpha, \beta) < p < m(\alpha, \beta).$ The Muckenhoupt's proof is based on the following inequality, which we will use later.

\begingroup
\setcounter{Th-Alfa}{\value{theorem}}
\setcounter{theorem}{1} 
\renewcommand\thetheorem{\Alph{theorem}}
\begin{theorem}{\cite[Th. 1]{Muck69}}\label{ThMuck}
Let $ \alpha,\beta > -1,$ $1<p<\infty$ and  $a, b \in\RR$ such that   \eqref{Muck-Cond} holds. Then there exists a constant, $C$, independent of $f$ and $n$, such that
\begin{equation*}
	\int_{-1}^{1} \left|\FSJn(f,x) (1-x)^{a}(1+x)^b\right|^p dx \leq C \int_{-1}^{1}   \left| f(x) (1-x)^{a}(1+x)^b\right|^p dx,
\end{equation*}
where $\FSJn(f,\cdot)$ is given by \eqref{FourCoef}-\eqref{FourPartialSum}.
\end{theorem}
\endgroup

Additionally, in \cite[Th. 2]{Muck69}, Muckenhoupt also proved that condition \eqref{Muck-Cond} is also necessary.

\section{Fourier series of Sobolev orthogonal polynomials}\label{Sect-SobolevPoly}

From now on, we will focus our attention on summarizing the main advances on the convergence of the Fourier-Sobolev series with respect to Sobolev norms of type-$p$ given by \eqref{Sobolev-pNorm}. The study of this convergence faces several problems. First of all, in general, the Sobolev  polynomials do not satisfy a three-term recurrence relation and, therefore, there is not a Christoffel-Darboux formula for their Dirichlet-Szeg\H{o} kernel. As a consequence, the argument outlined above can no be applied  to  this case. Additionally, since there are three different types of Sobolev inner products, we will need three different methods to study them separately. Let us consider the results that have been obtained so far in each case:

I.  \emph{Discrete case.}  In \cite{CiaMin18}, the authors  give  a complete characterization of the boundedness of the partial sum operators with respect to the inner product \eqref{SobolevProd} when $\ell=1$, $\mu_0 = \mu^{\alpha} + M(\delta_{-1} + \delta_1)$ where $\mu^{\alpha}$ is the Gegenbauer probability measure
		$$d\mu^{\scriptscriptstyle \!\alpha}(x) = \frac{\Gamma(2\alpha+2)}{2^{2\alpha+1} \Gamma^2(\alpha)} (1-x^2)^\alpha dx, \quad \alpha > -1/2,$$
and	$\mu_1 = N(\delta_{-1} + \delta_1)$. Due to the lack of a Christoffel-Darboux formula for Sobolev orthogonal polynomials, the main tools  involve  a special case of a general weighted transplantation theorem, see 	\cite[Th.  1.6]{Muck86} and the other one is a particular case of the multiplier result \cite[Th.  1.10]{Muck86}.
Using these tools, they prove (see  \cite[Th.  1.1]{CiaMin18}) the boundedness of the partial sum operators in the corresponding space for
		$$\frac{4(\alpha +1)}{2\alpha+3} < p < \frac{4(\alpha +1)}{2\alpha+1}.$$
Then, using the denseness of the polynomials in the corresponding Sobolev space (see \cite[Th. 1.2]{CiaMin18}) and the Banach-Steinhaus uniform boundedness principle, they  obtain that, for these values of $p$, the sequence of the partial sums of $f$, $\{S_n(f,\cdot)\}_{n\geq0}$, converges to the function $f,$ in the Sobolev norm $\normSp{\cdot}$ (see  \cite[Cor. 1.3]{CiaMin18}).

II.   \emph{Continuous case.}  The first article in this direction is due to F. Marcell\'{a}n, Y. Quintana and A. Urieles  \cite{MarceVickyUrieles}. They  consider the inner product
		$$\IpS{f}{g}= \int_{-1}^1 f(x) g(x) d\Jm(x) + \int_{-1}^1 f'(x) g'(x) d\fJm{\alpha+1}{\beta+1}(x).$$
	In  this case, it follows from \eqref{Jacobi-Der} that  the corresponding Jacobi-Sobolev polynomials are, up to a constant factor the classical Jacobi polynomials and, therefore, they satisfy a three term-recurrence relation. In the paper, the authors actually only focus on  the case of the Legendre-Sobolev polynomials, i.e. $\alpha=\beta =0$. The existence of a recurrence relation for them (see \cite[(3.35)]{MarceVickyUrieles}) implies that an explicit formula can be given for the Sobolev-Dirichlet-Szeg\H{o} kernel corresponding to the $n$-th partial sum. Unfortunately the corresponding part for the derivatives generates a hypersingular transform which is not easy to handle  without having to ask the function a lot of regularity. Consequently, the final result is not  entirely  satisfactory (see \cite[Th.  3.1]{MarceVickyUrieles}).
	
 The second article in this direction is due to  \'{O}.  Ciaurri and J. M\'{\i}nguez \cite{CiaMin19}. They  consider the  general case of the Fourier series of Jacobi-Sobolev polynomials, $\alpha, \beta >-1$, with  an arbitrary number order derivatives, i.e. orthogonal polynomials with respect to Sobolev-type inner product
	$$ \IpS{f}{g}= \sum_{k=0}^\ell \int_{-1}^1 f^{(k)}(x) g^{(k)}(x)\,  d\fJm{\alpha+k}{\beta+k}(x), \quad \alpha, \beta >-1, \quad \ell  \geq 1.
$$
Again, by \eqref{Jacobi-Der}, the corresponding Jacobi-Sobolev polynomials are multiples of the classical Jacobi polynomials. Because these polynomials faced similar problems as in \cite{MarceVickyUrieles}, they consider an alternative approach following basically what they had developed in \cite{CiaMin19}, applying  Muckenhoupt's multipliers and transplantation operators for Jacobi expansions  (see \cite[Th.  5.1, Th. 7.1 and Th. 1.10]{Muck86}, \cite{CiNoSt07})  and also Pollard's results \cite{Poll49}.

They prove in  \cite[Th.  1]{CiaMin19}, the boundedness of the partial sum operators $\FSSn(f,\cdot)$ in the corresponding space when
	$$\max \left\{\frac{4(\alpha +\ell+1)}{2(\alpha+\ell)+3},\frac{4(\beta +\ell+1)}{2(\beta+\ell)+3}\right\} < p < \min\left\{\frac{4(\alpha +				\ell+1)}{2(\alpha+\ell)+1},\frac{4(\beta +\ell+1)}{2(\beta+\ell)+1}\right\}.$$
Then, using the denseness of the polynomials in the corresponding Sobolev space, see \cite{RodAlvRomPestII} and the Banach-Steinhaus uniform boundedness principle, they  obtain that  this result  is equivalent to the convergence of  $\FSSn(f,\cdot)$ to $f$ in the Sobolev norm $\normSp{\cdot}$ (see  \cite[Corollary 2]{CiaMin19}).

Let $\left(d \mu_{0}, d \mu_{1}\right)$ be a pair of  positive Borel measures on the real line with finite moments of all orders and $P_{i,n}$ be the monic orthogonal polynomial of degree $n$ with respect to $d \mu_{i}.$   The pair   $\left(d \mu_{0}, d \mu_{1}\right)$ is called \emph{coherent} if there exists a sequence of nonzero real numbers $\left\{a_{n}\right\}_{n \geq 1}$ such that
$$
P_{1,n}\left(z\right)=\frac{P_{0,n+1}^{\prime}\left(z\right)}{n+1}+a_{n} \frac{P_{0,n}^{\prime}\left(z\right)}{n}, \quad n \geq 1.
$$
The notion of coherent pairs of measures was introduced in \cite{IKNS-S91} and it is closely related with the Sobolev inner product
\begin{equation}\label{Coherent-IP}
\IpS{f}{g}=\int_{a}^{b} f(x) g(x) d \mu_{0}(x)+\int_{a}^{b} f^{\prime}(x) g^{\prime}(x) d \mu_{1}(x),
\end{equation}
where $-\infty \leq a < b \leq \infty$. For a review of the study of the coherent pair of measures we recommended \cite[\S 5]{MarcellanXu}.  The study of convergence of Fourier series of orthogonal polynomials with respect to the inner product \eqref{Coherent-IP} has been carried out in \cite{CiaMin19-2}.

III:   \emph{Discrete-continuous case.}  I. I. Sharapudinov considered the case of the inner product
$$ \IpS{f}{g}= \sum_{k =0}^{\ell-1} f^{(k)} (-1) g^{(k)} (-1) + \int_{-1}^1 f^{(\ell)} (t)g^{(\ell)} (t) (1-t)^\alpha dt.$$
He proved that the associated Fourier series converges uniformly in $[-1,1]$ using that this case is a particular case of what he called mixed series of Jacobi polynomials considered by him in a previous paper.

In this case, as far as we know, a general result about convergence of the Fourier series in the corresponding Sobolev norm $\normSp{\cdot}$ is still unknown.

We will consider a more general case in this direction in the next section.

\section{Discrete-Continuous Jacobi-Sobolev Spaces}\label{Sect-CarSobSpac}

First, we need to consider the natural linear spaces for the problem. Let  $p\geq 1$, $\ell\in \NN$, $\alpha,\beta>-1$ and $\varpi=(\omega_0,\omega_1, \dots, \omega_{\ell-1})\in \RR^{\ell}$. Given a suitable function $f$, we define the  Jacobi-Sobolev norm as:
\begin{align}\label{NorSobp}
	\normSp{f}= \left(\sum_{k=0}^{\ell-1} \left|f^{(k)}(\omega_{k})\right|^{p} + \int_{-1}^{1} \left|f^{(\ell)}(x)\right|^{p} d\Jm(x)\right)^{\frac{1}{p}}.
\end{align}

It is clear that  $\normSp{\cdot}$ is a norm on the linear space of all polynomials $\PP$. We define the Jacobi-Sobolev  discrete-continuous  space $\WWp$ as the completion of  the normed space $(\PP,\normSp{\cdot})$. In other words, $\WWp$ is the collection of all equivalence classes of Cauchy sequences in $(\PP,\normSp{\cdot})$, under the following equivalence relation: $\{P_n\}_{n\geq0},$ $\{Q_n\}_{n\geq0}$ Cauchy sequences in  $(\PP,\normSp{\cdot})$ are equivalent $\{P_n\}_{n\geq0} \sim  \{Q_n\}_{n\geq0}$ if and only if $\lim_{n\to \infty}\normSp{P_n-Q_n}=0$.

On the other hand, let $\uup$ be the linear space of functions $f:\RR \longrightarrow \RR$ satisfying:
\begin{enumerate}
\item $f$ has derivative at $\omega_{k}$ of order $k$, $k=1,2, \dots,\ell-1$.
\item $f$ is $\ell-1$ times differentiable on $(-1,1)$ and $f^{(\ell-1)}$ is absolutely continuous on every compact subset of $(-1,1)$.
\item $f^{(\ell)}\in \JLp$.
\end{enumerate}
Obviously, \eqref{NorSobp} is a seminorm on $\uup$ and $\PP \subset \uup$. Let us write $f\sim g $ if and only if $\normSp{f-g}=0$. It is a straightforward result that this is an equivalence relation on $\uup$. Denote the linear space of the equivalence classes of this relation by  $\UUp$. Given $F\in \UUp$ and if $f$ is any element of $F$, then we define the norm in $\UUp$ as $\normSp{F}=\normSp{f}.$ Thus, $(\UUp,\normSp{\cdot})$ is a normed space.

In what follows,  we will introduce some notions which will be useful in the sequel. Given $\vec{x}=(x_0,x_1,\allowbreak\dots,x_{m-1})$, $\vec{y}=(y_0,y_1,\dots,y_{m-1}) \in \RR^m$, there exists a unique polynomial $\mathcal{A}_{\vec{x},\vec{y} }$ of degree at most $m-1$, such that
\begin{equation*}
	\mathcal{A}_{\vec{x},\vec{y}}^{(k)}( x_{k})=y_k, \quad k=0,1,\dots,m -1.
\end{equation*}
The polynomial $\mathcal{A}_{\vec{x},\vec{y} }$  is known   as the \emph{Abel-Goncharov interpolation polynomial}, associated with $\vec{y}$ on $\vec{x}$ (see \cite[Ch. 3]{AgWo93},  \cite[\S2.6]{Dav75}, or \cite[\S 8]{Whi35}). The Abel-Goncharov interpolation polynomial is a generalization of Taylor's polynomial, which corresponds to the case $ x_{m-1} = x_{m-2}=\cdots= x_0$.

 The polynomial $\mathcal{A}_{\vec{x},\vec{y}}$ is  given by the explicit expression
\begin{equation*}
   \mathcal{A}_{\vec{x},\vec{y}}(z)= \sum_{k=0}^{m-1} y_k \mathcal{G}_{\vec{x},k}(z),
\end{equation*}
where $\mathcal{G}_{\vec{x},0}\equiv 1$ and   $\mathcal{G}_{\vec{x},k}$, for each $k=1,\dots, m-1$,  is the polynomial of degree $k$, generated by the $k$-th iterated integral
\begin{equation*}
\mathcal{G}_{\vec{x},k}(s) = \int_{ x_{0} }^{s} \int_{ x_{1}}^{s_{1}} \! \cdots\!  \int_{ x_{k-2}}^{s_{k-2}}\!\!\int_{ x_{k-1}}^{s_{k-1}}  ds_{k}ds_{k-1}\!\cdots \!ds_{2} ds_{1}  .
\end{equation*}
   The polynomial $\mathcal{G}_{\vec{x},k}$ is called the \emph{$k$-th Goncharov's polynomial associated with }$\vec{x}$ and satisfies
   \begin{equation}\label{Goncharov_PolyDer}
       \mathcal{G}_{\vec{x},k}^{(\nu)} ( x_{\nu})= \funD{0}{\nu\neq k}{1}{\nu=k.}
   \end{equation}

\begin{proposition} \label{PropDensePoly}
	Let   $p \in [1,\infty)$ and  $F\in \UUp$, then there exists a sequence of polynomials $\{P_n\}_{n\geq0}$ such that
		$$\normSp{P_n-f}\longrightarrow 0,$$
	for every $f\in F$. As a consequence,   $\UUp$ can be identified as a normed subspace of \, $\WWp$.
\end{proposition}

\begin{proof}
	Let $f\in F\in \UUp$, then $f^{(\ell)}\in \JLp$. From the denseness of the linear space of continuous functions on  $ \JLp$ (see \cite[Theorem 3.14]{Rud87}) and  Weierstrass' approximation theorem  \cite[\S 4.1]{BorErd95}, for all $n\in \NN$, there exists a polynomial  $p_n$   such that
$$
\normJLp{f^{(\ell)}-p_{n}}<\frac{1}{n}.
$$
	Now, consider the vector $f(\varpi)=(f(\omega_0),f'(\omega_1),\dots, f^{(\ell-1)}(\omega_{\ell-1}))$ and the polynomial
\begin{align*}
P_{n}(s)=&\mathcal{A}_{\varpi,f({\varpi})}(s)+\int_{ \omega_0}^s\int_{ \omega_1}^{s_1}\cdots \int_{ \omega_{\ell-2}}^{s_{\ell-2}}\int_{ \omega_{\ell-1}}^{s_{\ell-1}}p_{n}(s_{\ell})ds_{\ell}ds_{\ell-1}\cdots ds_2ds_1.
\end{align*}
	It is clear that this polynomial satisfies $P_{n}^{(k)}(\omega_k)=f^{(k)}(\omega_k)$ for $k=0,1, \dots,\ell-1$, and $P_{n}^{(\ell)}\equiv p_{n}$. Thus  we have
$$
		\normSp{f-P_{n}}=\normJLp{f^{(\ell)}-P_{n}^{(\ell)}}= \normJLp{f^{(\ell)}-p_{n}}<\frac{1}{n},
$$
	and we get the first statement. Note that the sequence $\{P_{n}\}_{n\geq0}$ does not depend on the function $ f \in F$ chosen. Indeed, if  $h\sim f$, then
		$$\normSp{P_{n}-h}\leq\normSp{P_{n}-f}+\normSp{f-h}=\normSp{P_{n}-f}\longrightarrow 0.$$
	The identification  of $\UUp$ as a normed subspace of $\WWp$ follows from the fact that $\{P_{n}\}_{n\geq0}$ is a Cauchy sequence on ($\PP,\normSp{\cdot}$)
		$$\normSp{P_n-P_m} \leq  \normSp{P_n-f}+\normSp{f-P_m}.$$
	Then, if $F$ is the equivalence class of $f$ it can be identified with the equivalence class of $\{P_n\}_{n\geq0}$ in $\WWp$.
 \end{proof}
 From now on, we will not distinguish between the equivalent class $F$ and any representative of it $f$,  and we will   write $\UUp \subset\WWp$.  It is worth to remark that although $\UUp$ is the space of suitable functions  $f$ (in the sense of (2.) in the definition of $\UUp$) of minimum requirements such that $\normSp{f}$ makes sense, this is not broad enough to describe every element of $\WWp$, at least not for every combination of the parameters $p,\alpha,\beta,\ell$ and $\varpi$ (see Theorem \ref{Th-Completeness}).

In this section we obtain all the combinations of the values $p,\alpha,\beta$ and $\varpi$ for which $\WWp=\UUp$.
The study of completeness of Sobolev spaces for the general vectorial measures $\vec{\mu}=(\mu_0,\mu_1,\dots,\mu_{\ell})$, was given by J. M. Rodr\'{\i}guez et al. in \cite{RodAlvRomPestI,RodAlvRomPestII}. The Sobolev space defined in \cite{RodAlvRomPestI,RodAlvRomPestII}, in the case of our specific vectorial measure $\vec{\mu}=(\delta_{\omega_0},\allowbreak \delta_{\omega_1}, \dots,\delta_{\omega_{\ell-1}},\Jm)$, agrees with our space $\UUp$ in the essential case when $\omega_i\in[-1,1]$, for $i=0,1,\dots,\ell-1$ and $\UUp$ is complete, see Theorem \ref{Th-Completeness}.

We are going to prove the completeness of $\UUp$ by means of \cite[Th 5.1]{RodAlvRomPestI}. In this paper the authors introduce several definitions which we will follow in order to use the theorem mentioned above. We also prove that for the complementary combinations of the parameters $\alpha,\beta,\ell,\varpi$, the space $\UUp$ is not complete and, consequently, $\UUp\neq\WWp$.

\begin{definition}{(see \cite[Definition 1.4]{KuOp84}) and  \cite[Definition 2]{RodAlvRomPestI})}
Let $p\geq 1$ and  $K$ be a closed interval on the real line. A weight function $w$ defined on $K$ satisfies condition $\Bp(K)$ if $w^{-1}\in \Lp[\frac{1}{p-1}](\nu_K)$, where $\nu_K$ denotes the Lebesgue measure on $K$, and we write $ w\in \Bp(K)$.

If $J\subset \RR$ is any interval and $w$ is a weight function defined on $J$, then we say that $w\in \Bp(J)$ when $w\in \Bp(K)$ for every compact interval $K\subset J$. 
\end{definition}

\begin{example}\label{Ex-Bp}
Consider the Jacobi weight function $\Jw(x)=(1-x)^{\alpha}(1+x)^{\beta}$, with $\alpha,\beta>-1$. 
\begin{enumerate}
\item If $p=1$, then
	$$\Jw\in\funC{\Bp[1]([-1,1])}{\alpha,\beta\in (-1,0]}{\Bp[1]((-1,1])}{\alpha\in(-1,0], \beta\in (0,\infty)}{\Bp[1]([-1,1))}{\alpha\in (0,\infty), \beta\in (-1,0]}{\Bp[1]((-1,1))}{\alpha,\beta\in (0,\infty).}$$
\item If $p>1$, then
	$$\Jw\in\funC{\Bp([-1,1])}{\alpha,\beta\in (-1,p-1)}{\Bp((-1,1])}{\alpha\in (-1,p-1), \beta\in [p-1,\infty)}{\Bp([-1,1))}{\alpha\in [p-1,\infty), \beta\in (-1,p-1)}{\Bp((-1,1))}{\alpha,\beta\in [p-1,\infty).}$$
\end{enumerate}	
\end{example}

\begin{proposition}
Let $\alpha,\beta>-1,p\geq1$, $k=0,1,\dots,\ell$;  and consider the vectorial weight
\begin{align}\label{Def-JacWeightVector}
w(x)= & (w_0(x),w_1(x),\dots,w_{\ell-1}(x),w_\ell(x))=(0,0,\dots,0,\Jw(x)), \quad  x\in[-1,1],\\
\label{Def-Ik}
& \text{the interval } \quad  	I_k=\funD{(k \,p-1,(k+1)p-1]}{k=0,1,\dots,\ell-1}{(\ell\, p-1,\infty)}{k=\ell,}
\end{align}
and  the vectorial weight $\owj[]=(\owj[0],\owj[1],\dots,\owj[\ell])$, where
	$$\owj(x)=\funT{\mathbb{1}_{[-1,0]}(x)}{j=0,1,\dots,\ell-k-1}{(1+x)^{\beta-(\ell-j)p}\mathbb{1}_{[-1,0]}(x)}{j=\ell-k,\dots,\ell-1}{\Jw(x)}{j=\ell,}$$
and $\mathbb{1}_{A}$ denotes  the  characteristic function of the set $A \subset \RR$.

If $k$ is such that $\beta\in I_{k}$, then $\owj[]$ is a right completion of $w$ with respect to $-1$ (see \cite[Definition  5]{RodAlvRomPestI}). Analogously, we can found a left completion of $w$ with respect to $1$. Namely
	$$\owj(x)=\funT{\mathbb{1}_{[0,1]}(x)}{j=0,1,\dots,\ell-k-1}{(1-x)^{\alpha-(\ell-j)p}\mathbb{1}_{[0,1]}(x)}{j=\ell-k,\dots,\ell-1}{\Jw(x)}{j=\ell,}$$
where $\alpha\in I_k$.
\end{proposition}
\begin{proof}
It is straightforward to check $\owj[\ell]=\Jw=w_{\ell}$ and there exists $\epsilon=1$ such that $\owj=0=w_j$ for $x\notin [-1,-1+\epsilon]=[-1,0]$. Then it only remains to prove that $\twj=\owj-w_j=\owj$, $0\leq j\leq \ell-1$, satisfies
\begin{enumerate}
\item $\twj\in \Lp[1]([-1,0]).$
\item $\Lambda_{p}\left(\twj, \owj[j+1]\right)<\infty$ where $\dsty \Lambda_{p}(u, v) :=\sup_{\!-1<r<0}\left(\int_{-1}^{r} u\right)\left\|v^{-1}\right\|_{\Lp[1 /(p-1)]([r, -1+\epsilon])}$.
\end{enumerate}
The first statement comes from
\begin{align*}
 	\int_{-1}^0\twj(x)dx=&\int_{-1}^0\owj(x)dx \\
  = & \funD{1}{j=0,1,\dots,\ell-k-1}{\dsty \int_{-1}^0(1+x)^{\beta-(\ell-j)p}dx}{j=\ell-k,\dots,\ell-1,}<\infty,
\end{align*}
where the integral in the second case is finite because of $\beta-(\ell-j)p\geq\beta-kp>-1$. For the second statement we will split the proof in several cases.

\begin{itemize}[leftmargin=*]
\item Case $p>1$, $0\leq j\leq\ell-k-2$,
\begin{align*}
 \Lambda_{p}(\twj[j],\owj[j+1])=\sup _{-1<r<0}\left(\int_{-1}^{r} dx\right)\left(\int_r^0 dx\right)^{p-1}=\sup _{-1<r<0}(r+1)\left(-r\right)^{p-1}\leq1.
\end{align*}

\item Case $p>1$, $j=\ell-k-1$, $\beta<(k+1)p-1$,
		\begin{align*}
			\Lambda_{p}(\twj[\ell-k-1],\owj[\ell-k])\!\leq &\max\{\!1\!,\!2^{-\alpha}\}\!\sup _{-1<r<0}\left(\int_{-1}^{r} dx\right)\left(\int_r^0(1+x)^{-\frac{\beta-kp}{p-1}}dx\right)^{\!p-1}\\
				=&\frac{\max\{\!1\!,\!2^{-\alpha}\}(p\!-\!1)^{p-1}}{(\!-\beta\!+\!(k\!+\!1)p\!-\!1)^{p-1}} \!\sup _{-1<r<0}(r+1)\!\!\left(\!1\!-\!(1+r)^{\frac{-\beta+(k+1)p-1}{p-1}}\right)^{\!p-1}\\
				\leq&\frac{\max\{\!1\!,\!2^{-\alpha}\}(p-1)^{p-1}}{(\!-\beta\!+\!(k\!+\!1)p\!-\!1)^{p-1}},
		\end{align*}
		where we have used that $-\beta+(k+1)p-1>0$ in the last inequality and also to calculate the second integral.
		
\item Case $p>1$, $j=\ell-k-1$, $\beta=(k+1)p-1$,
		\begin{align*}
			\Lambda_{p}(\twj[\ell-k-1],\owj[\ell-k])\leq & \max\{\!1\!,\!2^{-\alpha}\} \!\sup _{-1<r<0}\left(\int_{-1}^{r} dx\right)\left(\int_r^0(1+x)^{-\frac{\beta-kp}{p-1}}dx\right)^{p-1}\\
				=&  \max\{\!1\!,\!2^{-\alpha}\} \! \sup _{-1<r<0}(r+1)\left(\int_r^0(1+x)^{-1}dx\right)^{p-1}\\ =& \sup _{-1<r<0}(r+1)\left(-\ln(r+1)\right)^{p-1}<\infty,		
		\end{align*}
where we have used that,  by L'H\^{o}pital rule, $\dsty \lim_{r\to -1^+}(1+r)(-\ln(1+r))^{p-1}=0.$

\item Case $p>1$, $\ell-k\leq j\leq \ell-1$,
		\begin{align*}
			\Lambda_{p}(\twj,\owj[j+1])\leq & \max\{\!1\!,\!2^{-\alpha}\} \! \sup _{-1<r<0}\int_{-1}^{r} \owj(x)dx\left(\int_r^0(1-x)^{-\frac{\beta-(\ell-j-1)p}{p-1}}dx\right)^{p-1}\\		
=&\frac{\max\{\!1\!,\!2^{-\alpha}\} \!(p-1)^{p-1}}{(\beta\!-\!(\ell\!-\!j)p\!+\!1)^p}\sup _{-1<r<0}\left(1-(1+r)^{\frac{\beta-(\ell-j)p+1}{p-1}}\right)^{p-1} \\
= & \frac{\max\{\!1\!,\!2^{-\alpha}\} \!(p-1)^{p-1}}{(\beta\!-\!(\ell\!-\!j)p\!+\!1)^p},
		\end{align*}
		where we have used that $\beta-(\ell-j)p+1>kp-(\ell-j)p\geq0$.
\item Case $p=1$, $0\leq j\leq\ell-k-2$,
\begin{align*}
 \Lambda_{p}(\twj[j],\owj[j+1])=  &  \Lambda_{p}(\owj[j],\owj[j+1])=\sup _{-1<r<0}\left(\int_{-1}^{r} dx\right)\left\|\mathbb{1}_{[r,0]}^{-1}\right\|_{\Lp[\infty]([r, 0])}\\
  = & \sup _{-1<r<0}(r+1)=1.
\end{align*}

\item Case $p=1$, $j=\ell-k-1$,
		\begin{align*}
			\Lambda_{p}(\twj[\ell-k-1],\owj[\ell-k])\leq& \max\{\!1\!,\!2^{-\alpha}\}\sup _{-1<r<0}\left\{\left(\int_{-1}^{r} dx\right)\sup_{r<x<0}(1+x)^{-(\beta-k)}\right\}\\ =&\max\{\!1\!,\!2^{-\alpha}\}\sup _{-1<r<0}(r+1)= \max\{\!1\!,\!2^{-\alpha}\},
		\end{align*}
		where we have used that $-\beta+k\geq-(k+1)p+1+k=0$.

\item Case $p=1$, $\ell-k\leq j\leq \ell-1$,
		\begin{align*}
			\Lambda_{p}(\twj,\owj[j+1])=&\Lambda_{p}(\owj,\owj[j+1])=\sup _{-1<r<0}\left(\int_{-1}^{r} \owj(x)dx\right)\left\|\left(\owj[j+1]\right)^{-1}\right\|_{\Lp[\infty]([r, 0])}\\
			&\leq \max\{\!1\!,\!2^{-\alpha}\} \!\sup _{-1<r<0}\left\{\int_{-1}^{r} \owj(x) dx\sup_{r<x<0}(1+x)^{-(\beta-(\ell-j)+1)}\right\}\\
			&=\frac{\max\{\!1\!,\!2^{-\alpha}\}}{\beta-(\ell-j)+1}\sup _{-1<r<0}\left\{(1+r)^{\beta-(\ell-j)+1}(1+r)^{-(\beta-(\ell-j)+1)}\right\}\\ & =\frac{\max\{\!1\!,\!2^{-\alpha}\}}{\beta-(\ell-j)+1},
		\end{align*}
		where we have used that $\beta-(\ell-j)+1>kp-(\ell-j)=k-(\ell-j)\geq0$.		
\end{itemize}
\end{proof}

\begin{proposition}\label{ProRegSets} Denote by $\Omega^{(k)}$ the set of $k$-regular points of the vectorial weight \eqref{Def-JacWeightVector}  (see \cite[Definition 6]{RodAlvRomPestI}).
	\begin{itemize}
		\item[$\bullet$] If $p>1$, then for $m=1,2,\dots,\ell$;
			\begin{align*}
				\Omega^{(\ell-m)}=\funC{[-1,1]}{\alpha,\beta\in (-1,mp-1)}{(-1,1]}{\alpha\in(-1,mp-1), \beta\in [mp-1,\infty)}{[-1,1)}{\alpha							\in[mp-1,	\infty), \beta\in (-1,mp-1)}{(-1,1)}{\alpha,\beta\in [mp-1,\infty).}
			\end{align*}
		\item[$\bullet$] If $p=1$, then for $m=1,2,\dots,\ell$;
			\begin{align*}
				\Omega^{(\ell-m)}=\funC{[-1,1]}{\alpha,\beta\in(-1,m-1]}{(-1,1]}{\alpha\in(-1,m-1], \beta\in (m-1,\infty)}{[-1,1)}{\alpha\in(m-1,\infty),\beta\in (-1,m-1]}{(-1,1)}{\alpha,\beta\in (m-1,\infty).}
			\end{align*}
	\end{itemize}
	
\end{proposition}
\begin{proof} From $\Jw\in\Bp((-1,1))$ it follows that $y$ is ($\ell-1$)-regular for every $y\in(-1,1)$ and the analysis of the regularity of $y=1$ and $y=-1$ are analogous. So we will only focus our attention on the point $y=-1$. First, we are going to prove that if $p>1$ and $\beta\in (-1,mp-1)\subset \cup_{i=0}^{m-1}I_{i}$,  where $I_i$ is as in \eqref{Def-Ik}, then $y=-1$ is an ($\ell-m$)-right regular point.

Take $0\leq k\leq m-1$ such that $\beta\in I_{k}$ .
\begin{itemize}[leftmargin=*]
\item Case $\beta\in (kp-1,(k+1)p-1)$,
\begin{align*}
\int_{-1}^0\left(\fowj[\ell-k]{k}\right)^{-\frac{1}{p-1}}dx&=\funD{\int_{-1}^0\left(x+1\right)^{-\frac{\beta-kp}{p-1}}dx}{k\geq 1}{\int_{-1}^0(1-x)^{-\frac{\alpha}{p-1}}(1+x)^{-\frac{\beta}{p-1}}dx}{k=0,}<\infty,
\end{align*}
where both integrals are finite due to  $\frac{\beta-kp}{p-1}<\frac{p-1}{p-1}=1$. Note that $\ell-k\geq\ell-m+1$.
\item Case $\beta=(k+1)p-1$,
\begin{align*}
\int_{-1}^0\left(\owj[\ell-(k+1)]\right)^{-\frac{1}{p-1}}dx =1.
\end{align*}
Note that, as $\beta<mp-1$, we get $k\leq m-2$, so $\ell-(k+1)\geq\ell-m+1.$
\end{itemize}
In both cases, we have proved that there exists $\ell-m<j\leq \ell$ such that $\owj \in\Bp([-1,0])$, so $y=-1$ is $(\ell-m)$-right regular.

Now, consider $p=1$, and $\beta\in (-1,m-1]=\cup_{k=0}^{m-1}I_k$.
Take $0\leq k\leq m-1$ such that $\beta\in I_{k}$.
\begin{align*}
\sup_{-1\leq x\leq 0}\left\{\left(\fowj[\ell-k]{k}(x)\right)^{-1}\right\}&=\funD{\dsty \sup_{-1\leq x\leq 0}\left\{\left(x+1\right)^{-\beta+k}\right\}}{k\geq 1}{\dsty \sup_{-1\leq x\leq 0}\left\{(1-x)^{-\alpha}(1+x)^{-\beta}\right\}}{k=0,}<\infty.
\end{align*}
Note that $\ell-k\geq\ell-m+1$, so we have proved that there exists $\ell-m<j\leq \ell$ such that $\owj\in\Bp([-1,0])$, so $y=-1$ is ($\ell-m$)-right regular.

Now, we are going to prove that in the other case the point $y=-1$ is not ($\ell-m$)-right regular. But first, we will prove the following auxiliary result. If $w$ is a non-negative Lebesgue measurable function, then
\begin{align}\label{AxResult}
\normLp[\frac{1}{p-1}]{[a,b]}{w^{-1}}\geq \frac{(b-a)^{p}}{\int_{a}^bw(x)dx}, \quad a<b.
\end{align}
This inequality comes from Holder inequality
\begin{align*}
(b-a)^p&=\left(\int_{a}^bdx\right)^{p}=\left(\int_{a}^bw^{\frac{1}{p}}w^{-\frac{1}{p}}dx\right)^{p}\leq \left(\normLp{[a,b]}{w^{\frac{1}{p}}}\normLp[\frac{p}{p-1}]{[a,b]}{w^{-\frac{1}{p}}}\right)^p\\
&=\normLp[1]{[a,b]}{w}\normLp[\frac{1}{p-1}]{[a,b]}{w^{-1}}.
\end{align*}
Let $\ow$ be a right completion of $w$ with respect to $y=-1$. Then, we will prove by induction that for $0<s<\frac{\epsilon}{\ell+1}$, (where $\epsilon$  is the given one for $\ow$ in the definition of right completion) we have
\begin{align}\label{Ineq-noReg}
\normLp[\frac{1}{p-1}]{J_{k,s}}{\ow_{\ell-k}^{-1}}\geq M_ks^{-\beta+(k+1)p-1}, \quad k=0,1,\dots,\ell,
\end{align}
where $J_{k,s}=[-1+(\ell-k)s,-1+(\ell-k+1)s]$ and $M_k$ is some positive constant independent of $s$.
For $k=0$ we have
\begin{itemize}
\item Case $p>1$, $\beta\in [mp-1,\infty)$,
		\begin{align*}
			\normLp[\frac{1}{p-1}]{J_{0,s}}{\ow_{\ell}^{-1}}=&\left(\int_{-1+\ell s}^{-1+(\ell+1)s}(1-x)^{-\frac{\alpha}{p-1}}(1+x)^{-\frac{\beta}{p-1}}dx\right)^{p-1}\\
				\geq &\min\!\left\{\!\!\frac{1}{2^{\alpha}},\!\frac{1}{(2-\epsilon)^{\alpha}}\!\!\right\}\left(\int_{-1+\ell s}^{-1+(\ell+1)s}(1+x)^{-\frac{\beta}{p-1}}dx\right)^{p-1}\\=&M_0s^{-\beta+p-1},
		\end{align*}
		where
\begin{align*}
M_0=&\!\funD{\min\!\left\{\!\!\frac{1}{2^{\alpha}},\!\frac{1}{(2-\epsilon)^{\alpha}}\!\!\right\}\frac{(p-1)^{p-1}}{(\beta-p+1)^{p-1}}\!
\left(\ell^{-\frac{\beta}{p-1}}-(\ell+1)^{{-\frac{\beta}{p-1}}}\right)^{p-1}\!\!}{\beta>p-1}{\min\!\left\{\!\!\frac{1}{2^{\alpha}},
\!\frac{1}{(2-\epsilon)^{\alpha}}\!\!\right\}\left[\ln\left(\frac{\ell+1}{\ell}\right)\right]^{p-1}}{\beta=p-1,}>0.
\end{align*}		
\item Case $p=1$, $\beta\in (m-1,\infty)$,
		\begin{align*}
			\normLp[\infty]{J_{0,s}}{\ow_{\ell}^{-1}}&=\sup_{-1+\ell s<x<-1+(\ell+1)s}\left\{(1-x)^{-\alpha}(1+x)^{-\beta}\right\}\\
				&\geq\min\{2^{-\alpha},(2-\epsilon)^{-\alpha}\}\sup_{-1+\ell s<x<-1+(\ell+1)s}\left\{(1+x)^{-\beta}\right\} =M_0s^{-\beta},
		\end{align*}
		where $M_0=\min\{2^{-\alpha},(2-\epsilon)^{-\alpha}\}\ell^{-\beta}$.
\end{itemize}

Now, suppose that \eqref{Ineq-noReg} is true for $k-1$. Then from \eqref{AxResult} it follows
\begin{align*}
\normLp[\frac{1}{p-1}]{J_{k,s}}{\ow_{\ell-k}^{-1}}&\geq \frac{s^p}{\int_{-1+(\ell-k)s}^{-1+(\ell-k+1)s}\ow_{\ell-k}(x)dx}\\
	&\geq \frac{s^p\normLp[\frac{1}{p-1}]{J_{k-1,s}}{\ow_{\ell-k+1}^{-1}}}{\int_{-1}^{-1+(\ell-k+1)s}\tw_{\ell-k}(x)dx\normLp[\frac{1}{p-1}]{[-1+(\ell-k+1)s,-1+\epsilon]}{\ow_{\ell-k+1}^{-1}}}\\
	&\geq \frac{s^p\normLp[\frac{1}{p-1}]{J_{k-1,s}}{\ow_{\ell-k+1}^{-1}}}{\Lambda_{p}(\tw_{\ell-k},\ow_{\ell-k+1})} \geq \frac{M_{k-1}s^{-\beta+(k+1)p-1}}{\Lambda_{p}(\tw_{\ell-k},\ow_{\ell-k+1})}.
\end{align*}
Taking $M_{k}=\frac{M_{k-1}}{\Lambda_{p}(\tw_{\ell-k},\ow_{\ell-k+1})}>0$ we conclude the proof of \eqref{Ineq-noReg}. Thus we have
\begin{align*}
\normLp[\frac{1}{p-1}]{J_{\ell-j,s}}{\ow_{j}^{-1}}\geq M_{\ell-j}s^{-\beta+(\ell-j+1)p-1}, \quad j=0,1,\dots,\ell,
\end{align*}
and from here it follows
\begin{itemize}[leftmargin=*]
\item Case $p>1$, $\beta\in [mp-1,\infty)$,
\begin{align*}
\lim_{s\to 0}\normLp[\frac{1}{p-1}]{[-1+js,-1+(j+1)s]}{\ow_{j}^{-1}}\neq 0, \quad j=\ell-m+1,\dots,\ell.
\end{align*}
\item Case $p=1$, $\beta\in (m-1,\infty)$,
\begin{align*}
 \lim_{s\to 0}\normLp[\infty]{[-1+(\ell-j)s,-1+(\ell-j+1)s]}{\ow_{j}^{-1}}=\infty, \quad j=\ell-m+1,\dots,\ell.
\end{align*}
\end{itemize}
Thus $\ow_{j}\notin \Bp([-1,-1+\epsilon])$ for $j=\ell-m+1,\dots,\ell$, so $y=-1$ is not $(\ell-m)$-right regular.
 \end{proof}

\begin{proposition}\label{Pro-type2}
Let $\vec{\mu}_{[-1,1]}=(\mu_0,\dots,\mu_{\ell-1},\Jm)$ be a vectorial measure, where
	$$\mu_j=\funD{\delta_{\omega_j}}{\omega_j\in[-1,1]}{0}{\omega_j\notin[-1,1]},\quad j=0,1,\dots,\ell-1.$$
Then $\vec{\mu}_{[-1,1]}$ is a vectorial measure in $[-1,1]$ of type 2 (see \cite[Definition  11]{RodAlvRomPestII}) if and only if 	$ \dsty \omega_{j}\in\Omega^{(j)}$, for every $j$ such that $\omega_j\in[-1,1].$

Recall that $\Omega^{(k)}$ denotes the set of $k$-regular points of the vectorial weight \eqref{Def-JacWeightVector} given explicitly in Proposition \ref{ProRegSets}.
\end{proposition}
\begin{proof}
The condition $\omega_{j}\in\Omega^{(j)}$ is just the condition in the definition of strongly $p$-admissible vectorial measure (see \cite[Definition 8]{RodAlvRomPestI}) and $\vec{\mu}_{[-1,1]}$ is finite (each component measure is finite). Then, we only have to check that there exist real numbers $-1\leq a_1<a_2<a_3<a_4\leq 1$ such that
\begin{enumerate}
\item $w_\ell\in\Bp([a_1,a_4])$,
\item if $-1<a_1$, then $w_j$ is comparable (see \cite[Definition 1]{RodAlvRomPestII}) to a non-decreasing weight in $[-1,a_2]$, for $0\leq j \leq \ell$,
\item if $a_4<1$, then $w_j$ is comparable (see \cite[Definition 1]{RodAlvRomPestII}) to a non-decreasing weight in $[a_3,1]$, for $0\leq j \leq \ell$,
\end{enumerate}
where $d\mu_j=d(\mu_j)_s+w_jdx$, $j=0,1,\cdots,\ell$ and $(\mu_j)_s$ is singular with respect to Lebesgue measure. Therefore, we obtain $w_j\equiv 0$, $j=0,1,\dots,\ell-1$;  and $w_{\ell}(x)=\Jw(x)$. Since $w^{\prime}_{\ell}(x)=(1-x)^{\beta-1}(1+x)^{\alpha-1}(\beta-\alpha-(\beta+\alpha)x)$ it follows that there exists $0<\epsilon<1$ (depending on $\alpha$ and $\beta$) such that
$$w^{\prime}(x)  \funD{>0\quad \text{ for }x\in(-1,-1+\epsilon]}{\beta>0}{<0 \quad \text{ for } x\in[1-\epsilon,1)}{\alpha>0.}$$
Thus, from Example \ref{Ex-Bp}, the three conditions are easy verified taking $a_2=-1+\epsilon$,  $a_3=1-\epsilon$ and
			\begin{align*}
				\text{$\bullet$ Case $p>1$ } \;& \funC{a_1=-1,a_4=1}{\alpha,\beta\in (-1,p-1)}{a_1=-1+\frac{\epsilon}{2},a_4=1}{\alpha\in(-1,p-1), \beta\in [p-1,\infty)}{a_1=-1,a_4=1-\frac{\epsilon}{2}}{\alpha	\in[p-1,	\infty), \beta\in (-1,p-1)}{a_1=-1+\frac{\epsilon}{2},a_4=1-\frac{\epsilon}{2}}{\alpha,\beta\in [p-1,\infty).}\\
			\text{$\bullet$ Case $p=1$ } \;	& \funC{a_1=-1,a_4=1}{\alpha,\beta \in(-1,0]}{a_1=-1+\frac{\epsilon}{2},a_4=1}{\alpha\in(-1,0], \beta\in (0,\infty)}{a_1=-1,a_4=1-\frac{\epsilon}{2}}{\alpha\in(0,\infty),\beta\in (-1,0]}{a_1=-1+\frac{\epsilon}{2},a_4=1-\frac{\epsilon}{2}}{\alpha,\beta\in (0,\infty).}
			\end{align*}
 \end{proof}

\begin{proposition}
	If $m\in\ZZp$, then there exist $m+1$ positive real numbers $\{\lambda_{m,k}\}_{k=0}^m $ such that
	\begin{align} \label{For-DerLn}
		\left[\frac{1}{(1\pm x)\ln(1\pm x)}\right]^{(m)}=\frac{(\mp1)^{m}}{(1\pm x)^{m+1}\ln(1\pm x)}\sum_{k=0}^{m}\frac{\lambda_{m,k}}{\ln^{k}(1\pm x)}.
	\end{align}
\end{proposition}
\begin{proof}
	We are going to prove it by induction on $m$. The case $m=0$ is trivial. Suppose the formula \eqref{For-DerLn} holds for $m=n$ and we are going to prove it for $m=n+1$.
	\begin{align*}
		\left(\frac{1}{(1\pm x)\ln(1\pm x)}\right)^{\!\!(n+1)}\!=& \left(\frac{(\mp1)^n}{(1\pm x)^{n+1}\ln(1\pm x)}\sum_{k=0}^{n}\frac{\lambda_{n,k}}{\ln^{k}(1\pm x)}\right)^{\prime}\\
			=& (\mp1)^n\sum_{k=0}^{n}\left(\frac{\lambda_{n,k}}{(1\pm x)^{n+1}\ln^{k+1}(1\pm x)}\right)^{\prime}\\
			=& \frac{(\mp1)^{n+1}}{(1\pm x)^{n+2}\ln(1\pm x)}\sum_{k=0}^{n+1}\frac{(n+1)\lambda_{n,k}+k\lambda_{n,k-1}}{\ln^{k}(1\pm x)},
	\end{align*}
	where $\lambda_{n,-1}=\lambda_{n,n+1}=0$. Now the proof is completed taking  $\lambda_{n+1,k}=(n+1)\lambda_{n,k}+k\lambda_{n,k-1},$ for $ k=0,1,\dots ,n+1.$
 \end{proof}

\begin{proposition}\label{Pro-LpDifNoNull}
	Let $p\geq 1$, $m\in\NN$, $\beta>-1$, $a\in(0,1)$ and
		$$\alpha\funD{\geq mp-1}{p>1}{> mp-1}{p=1,}$$
	then  the functions
	\begin{align*}
		\phi_m(x)&=\funD{\left[\ln\left(\ln\left(\frac{1}{1-x}\right)\right)\right]^{(m)}}{\alpha=mp-1}{\left[\ln\left(\frac{1}{1-x}\right)\right]^{(m)}}{\alpha>mp-1,}\quad x\in[a,1),\\
		\phi_{-m}(x)&=\funD{\left[\ln\left(\ln\left(\frac{1}{1+x}\right)\right)\right]^{(m)}}{\alpha=mp-1}{\left[\ln\left(\frac{1}{1+x}\right)\right]^{(m)}}{\alpha>mp-1,}\quad x\in(-1,-a],
	\end{align*}
	are continuous, satisfy $\mathbb{1}_{[a,1)}\phi_m\in\Lp(\Jm)$, $\mathbb{1}_{(-1,-a]}\phi_{-m}\in\Lp(\Jm)$ and
	\begin{align*}
		\lim_{x\to \pm 1}\int_{\pm a}^x\int_{\pm a}^{x_1}\cdots\int_{\pm a}^{x_{m-1}}\phi_{\pm m}(x_{m})dx_{m}dx_{m-1}\cdots dx_1=\infty
	\end{align*}
\end{proposition}
\begin{proof}
	The first and the last statements are straightforward and the proofs of $\mathbb{1}_{[a,1)}\phi_m\in\Lp(\Jm)$ and $\mathbb{1}_{(-1,-a]}\phi_{-m}\in\Lp(\Jm)$ are analogous, so we only going to prove  $\mathbb{1}_{[a,1)}\phi_m\in\Lp(\Jm)$.
Using \eqref{For-DerLn} we obtain
\begin{itemize}[leftmargin=*]
\item Case $\alpha=mp-1$, $p>1$,
		\begin{align*}
			\normLp{\!\Jm}{\mathbb{1}_{[a,1)}\phi_m} ^p\!\!=\!&\int_{a}^1\left|\left[\ln\left(\ln\left(\frac{1}{1-x}\right)\right)\right]^{(m)}\right|^p(1-x)^{mp-1}(1+x)^{\beta}dx\\
				=&\int_{a}^1\left|\frac{1}{(1-x)^{m}\ln(1-x)}\sum_{k=0}^{m-1}\frac{\lambda_{m-1,k}}{\ln^{k}(1-x)}\right|^p(1-x)^{mp-1}(1+x)^{\beta}dx\\
\leq &\max\{\!(1\!+\!a)^\beta\!,2^{\beta}\}\! \int_{a}^1\!\!\!\frac{1}{\ln^{p}\!\!\left(\!\frac{1}{1-x}\right)}\!\left| \sum_{k=0}^{m-1}\frac{(-1)^{k}\lambda_{m-1,k}}{\ln^{k}\left(\!\frac{1}{1-x}\right)}\right|^p\!\!d\ln\!\left(\!\!\frac{1}{1-x}\!\right)\\
\leq &\max\{\!(1\!+\!a)^\beta\!,2^{\beta}\}\!\int_{\ln\left(\frac{1}{1-a}\right)}^{\infty}\frac{1}{y^p}\left|\sum_{k=0}^{m-1}\frac{(-1)^{k}\lambda_{m-1,k}}{y^{k}}\right|^p\!dy<\infty,
		\end{align*}
		where the last integral is convergent by the limit comparison test
		 $$\lim_{y\to\infty}\frac{\frac{1}{y^p}\left|\sum_{k=0}^{m-1}\frac{(-1)^{k}\lambda_{m-1,k}}{y^{k}}\right|^p}{\frac{1}{y^p}}=\lambda_{m-1,0}^p>0.$$
	
\item Case $\alpha>mp-1$, $p\geq 1$,
		\begin{align*}
			\normLp{\!\Jm}{\mathbb{1}_{[a,1)}\phi_m} ^p\!\!=&\int_{a}^1\left|\left[\ln\left(\frac{1}{1-x}\right)\right]^{(m)}\right|^p(1-x)^{\alpha}(1+x)^{\beta}dx\\
\leq &  \frac{\max\{(1+a)^\beta,2^{\beta}\}}{\alpha-mp+1}\left(1-a\right)^{\alpha-mp+1}.
		\end{align*}
\end{itemize}
 \end{proof}

We are ready to prove the completeness of the spaces $\UUp$.
\begin{theorem} \label{Th-Completeness}
Let $p\geq 1$, $\alpha,\beta>-1$. Then $\UUp$ is a complete space if and only if one of the following statements holds
	\begin{enumerate}
	    \item[$\bullet$] $\ell=1$.
		\item[$\bullet$]  $\ell\geq 2$, $p>1$ and for $m=1,\dots,\ell-1$,
			\begin{align*}
				\omega_{\ell-m}\in\funC{\RR}{\alpha,\beta\in (-1,mp-1)}{\RR\setminus\{-1\}}{\alpha\in(-1,mp-1), \beta\in [mp-1,\infty)}{\RR\setminus\{1\}}{\alpha							\in[mp-1,	\infty), \beta\in (-1,mp-1)}{\RR\setminus\{-1,1\}}{\alpha,\beta\in [mp-1,\infty).}
			\end{align*}
		\item[$\bullet$]  $\ell\geq 2$, $p=1$,  and for $m=1,\dots,\ell-1; $
			\begin{align*}
				\omega_{\ell-m}\in\funC{\RR}{\alpha,\beta \in(-1,m-1]}{\RR\setminus\{-1\}}{\alpha\in(-1,m-1], \beta\in (m-1,\infty)}{\RR\setminus\{1\}}{\alpha\in(m-1,\infty),\beta\in (-1,m-1]}{\RR\setminus\{-1,1\}}{\alpha,\beta\in (m-1,\infty).}
			\end{align*}
	\end{enumerate}
	In addition, in the affirmative case we obtain $\UUp=\WWp$.
\end{theorem}
Note that there is not restrictions over $\omega_0$ in any case.
\begin{proof}
Let $\{f_n\}_{n\geq0}\subset\UUp$ be a Cauchy sequence. From Proposition \ref{PropDensePoly} there exists a sequence of polynomials $\{P_n\}_{n\geq0}$ such that $\normSp{P_n-f_n}<\frac{1}{n}.$

Consider the space $V^{k,p}([-1,1],\vec{\mu}_{[-1,1]})$ defined in \cite[Definition 9]{RodAlvRomPestII}, where the measure $\vec{\mu}_{[-1,1]}$ is as in Proposition \ref{Pro-type2} and whose norm is given by
\begin{align*}
\qquad	\|f\|_{\tilde{\sob},p}= \left(\sum_{\substack{k=0\\\omega_k\in[-1,1]}}^{\ell-1}\left|f^{(k)}(\omega_{k})\right|^{p} + \int_{-1}^{1} \left|f^{(\ell)}(x)\right|^{p} d\Jm(x)\right)^{\frac{1}{p}}.
\end{align*}
From Proposition \ref{Pro-type2}, \cite[Remark 4 in Definition 16]{RodAlvRomPestI} and \cite[Theorem 5.1]{RodAlvRomPestI} this space is complete. So, as the restriction of $P_n$ to $[-1,1]$ belongs to $V^{k,p}([-1,1],\vec{\mu}_{[-1,1]})$ and
$$\|P_{n_1}-P_{n_2}\|_{\tilde{\sob},p}\leq \normSp{P_{n_1}-P_{n_2}}\leq \frac{1}{n_1}+\normSp{f_{n_1}-f_{n_2}}+\frac{1}{n_2}, $$
there exists a function $\tilde{f}\in V^{k,p}([-1,1],\vec{\mu}_{[-1,1]})$ such that
\begin{align}\label{For-CompleProof}
\|P_n-\tilde{f}\|_{\tilde{\sob},p}\longrightarrow 0.
\end{align}
On the other hand, since
\begin{align*}
\left|P_{n_1}^{(j)}(\omega_{j})-P_{n_2}^{(j)}(\omega_{j})\right|^p\leq &\sum_{k=0}^{\ell-1} \left|P_{n_1}^{(k)}(\omega_{k})-P_{n_2}^{(k)}(\omega_{k})\right|^p  \\ & + \int_{-1}^{1} \left|P_{n_1}^{(\ell)}(x)-P_{n_2}^{(\ell)}(x)\right|^p d\Jm(x)= \normSp{P_{n_1}-P_{n_2}}^p,
\end{align*}
it follows that $\{P_n^{(j)}(\omega_{j})\}_{n\geq 0}$, $j=0,1,\dots,\ell-1$ is a Cauchy sequence in $\RR$, then there exist values $f^{k}_{\omega_k}$ such that $ P_n^{(k)}(\omega_{k})\to f^{k}_{\omega_k}$, $k=0,1,\dots,\ell-1$.

Note that, from $\eqref{For-CompleProof}$ we obtain $\tilde{f}^{(k)}(\omega_k)=f^{k}_{\omega_k}$, for $\omega_k\in [-1,1]$,  where the derivative of $\tilde{f}$ at $\pm 1$, as is usual, means the corresponding usual side derivative.

Consider the vector ${f}(\varpi)=(f^{0}_{\omega_0},f^{1}_{\omega_1},\dots,f^{\ell-1}_{\omega_{\ell-1}})$ and the function
	$$f(x)=\funD{\tilde{f}(x)}{x\in[-1,1]}{\mathcal{A}_{\varpi,f(\varpi)}(x)}{x\in \RR\setminus [-1,1],}$$
Since $\tilde{f}^{(\ell-1)}$ is absolutely continuous on compact subsets of $\Omega^{(\ell-1)}\supset (-1,1)$ (see \cite[Def. 9]{RodAlvRomPestII}), the function $f$ belongs to $\UUp$ and satisfies
\begin{align*}
\normSp{P_n-f}^p&=\sum_{\omega_k\notin[-1,1]}\left|P^{(k)}_n(\omega_k)-f^{(k)}(\omega_k)\right|^p+\|P_n-f\|_{\tilde{\sob},p}^p\\
		&=\sum_{\omega_k\notin[-1,1]}\left|P^{(k)}_n(\omega_k)-f^{k}_{\omega_k}\right|^p+\|P_n-\tilde{f}\|_{\tilde{\sob},p}^p\longrightarrow 0.
\end{align*}

Finally, from $\normSp{f_n-f}\leq\normSp{f_n-P_n}+\normSp{P_n-f}\longrightarrow 0,$
we conclude that $\UUp$ is a complete space. Thus, as $\PP\subset \UUp\subset\WWp$ and $\WWp$ is the completion of $\PP$ we obtain $\WWp=\UUp$.

The second part of the proof is to show that the conditions over $\omega_k$ are also necessary. Then $\ell\geq 2$ and let us consider the case $\omega_{\ell-m}=1$ ($\alpha\geq mp-1$ if $p>1$ or $\alpha>m-1$ if $p=1$), the other case, when $\omega_{\ell-m}=-1$, is analogous. Take the sequence of functions
	\begin{align*}
		f_n(x)\!=\int_{c_0}^x\int_{c_1}^{x_1}\!\cdots\!\int_{c_{\ell-1}}^{x_{\ell1}}\!\!\mathbb{1}_{[a,a_n]}\!(x_{\ell})\phi_{\eta }(x_{\ell})dx_{\ell}dx_{\ell-1}\cdots dx_1, \quad n\geq 1,
	\end{align*}
	where  $\phi_m$ is the function defined in Proposition \ref{Pro-LpDifNoNull}, $\{a_n\}_{n=0}^{\infty}\subset \RR$ is an increasing sequence with $1$ as limit and $a:=a_0>0$. Since $\mathbb{1}_{[a,a_n]}\phi_{\eta} \in\Lp(\Jm)$ it follows that $f_n^{(\ell)}\in\JLp,$	so $f_n\in \UUp$ for $n\in\NN$.  Note also that, $f^{(k)}_n(c_k)=0$ for $n\geq 1$, $k=0,1,\dots,\ell-1$. Now we are going to prove that $f_n$ is a Cauchy sequence
	\begin{align*}
		\normSp{f_{n_1}-f_{n_2}}^p &=\sum_{k=0}^{\ell-1}\left|f^{(k)}_{n_1}(\omega_{k})-f^{(k)}_{n_2}(\omega_{k})\right|^p + \int_{-1}^{1} \left|f^{(\ell)}_{n_1}(x)-f_{n_2}^{(\ell)}(x)\right|^p d\Jm(x),\\
			&=\left|\int_{a_{n_2}}^{a_{n_1}} |\phi_m(x)|^pd\Jm(x)\right|.
	\end{align*}
	Since $\mathbb{1}_{[a,1)}\phi_m\in\Lp(\Jm)$, the sequence $\int_{a_n}^{1} |\phi_m(x)|^p d\Jm(x)$ tends to $0$ when $n \to\infty$, so it is a
	Cauchy sequence on $\RR$. Consequently, $f_n$ is a Cauchy sequence on $\UUp$. On the other hand, if $f\in\UUp$ is a function such that $f_n\to f$ in
	$\UUp$, then
	\begin{align*}
		\sum_{k=0}^{\ell-1}\left|f^{(k)}_n(\omega_{k})-f^{(k)}(\omega_{k})\right|^p + \int_{-1}^{1} \left|f^{(\ell)}_n(x)-f^{(\ell)}(x)\right|^p
		d\Jm(x)=\normSp{f_n-f}\longrightarrow 0.
	\end{align*}
	Hence, $f_n^{(k)}(\omega_k)\longrightarrow f^{(k)}(\omega_k)$ for $k=0,1,\cdots,\ell-1$ and
	\begin{align*}
		\int_{a}^{b_n} \left|\phi_m(x)-f^{(\ell)}(x)\right|^p d\Jm(x)\longrightarrow 0.
	\end{align*}
	Since $f\in\UUp$ we get $\mathbb{1}_{[a,1)}\phi_m-f^{(\ell)}\in\Lp(\Jm)$, so $$\int_{a}^{1} \left|\phi_m(x)-f^{(\ell)}(x)\right|^p d\Jm(x)=0,$$  and, as a consequence $f^{(\ell)}\equiv\phi_m$ a.e. on $[a,1]$.
	Thus we can write for $x\in[a,1)$
	\begin{align*}
		f^{(\ell-m)}(x)=\int_{a}^x\int_{a}^{x_1}\cdots\int_{a}^{x_{m-1}}\phi_{m}(x_{m})dx_{m}dx_{m-1}\cdots dx_1+Q_{m-1}(x), \quad x\in [a,1),
	\end{align*}
	where $Q_{m-1}$ is a polynomial of degree at most $m-1$.
	Now using that $\omega_{\ell-m}=1$ we obtain that $f$ has derivative of order $\ell-m$ on $[a,1]$ and the function
	$f^{(\ell-m)}$ is the derivative of $f^{(\ell-m-1)}$ on the closed interval $[a,1]$ (note that $m\leq \ell-1$). From Darboux's theorem
	\cite [Th. 5.12]{Rud76} it follows that	$f^{(\ell-m)}([1-\epsilon,1])$ is an interval for every sufficiently small 	$\epsilon>0$ and we 		get the contradiction $\lim_{x\to 1^{-}}f^{(\ell-m)}(x)=\infty$ (see Proposition \ref{Pro-LpDifNoNull}).
 \end{proof}

\begin{corollary}\label{CoroFunc}\

\begin{enumerate}
\item[$\bullet$] If $\ell=1$ and $ p\geq 1$, then $\dsty \mathbf{W}^{\scriptscriptstyle\alpha,\beta}_{\scriptscriptstyle \omega_0,p}=\mathbf{U}^{\scriptscriptstyle \alpha,\beta}_{\scriptscriptstyle\omega_0,p}$,    $\forall \omega_0\in\RR.$
\item[$\bullet$] If  $\ell\geq 2$ and $p>1$, then
$$\dsty \WWp=\UUp,  \ \forall \varpi\in\RR^\ell,  \; \text{ if and only if } \; \alpha,\beta\in (-1,p-1).$$
\item[$\bullet$] If $\ell\geq 2$ and $p=1$,  then $$\WWp=\UUp, \;  \forall\varpi\in\RR^\ell,  \; \text{ if and only if } \; \alpha,\beta\in (-1,0].$$
\end{enumerate}
\end{corollary}

\section{Sobolev-type orthogonal polynomials and Fourier series}

On $\PP$, let us consider the Sobolev-type  inner product, related to the norm (\ref{NorSobp}) with $p=2$
\begin{equation}\label{Sobolev-IP}
   \IpS{f}{g}:= \sum_{k=0}^{\ell-1} f^{(k)}(\omega_{k}) \, g^{(k)}(\omega_{k}) + \int_{-1}^{1} f^{(\ell)}(x) \,g^{(\ell)}(x) d\Jm(x).
\end{equation}

Let $\PRn$  be the $\ell$-th iterated integral  of $\pn$ (the $n$-th Jacobi orthonormal polynomial, as in \eqref{OrthNormal-Jacobi}),  normalized by the conditions
 \begin{equation}\label{AG-conditions}
\left(\PRn\right)^{(k)}(\omega_{k})=0, \quad k=0,1,\dots, \ell-1,
 \end{equation}
 or equivalently,
 \begin{equation*}
  \PRn(s) =\,\int_{\omega_{0} }^{s} \int_{\omega_{1}}^{s_{1}} \cdots \int_{\omega_{\ell-2}}^{s_{\ell-2}} \int_{\omega_{\ell-1}}^{s_{\ell-1}}  \pn(s_{\ell}) ds_{\ell}ds_{\ell-1}\cdots ds_{2}ds_{1}  .
 \end{equation*}


\begin{theorem} The polynomials
	\begin{equation}\label{MonicSOP}
		\sn(x)=\funD{\mathcal{G}_{\varpi,n}(x)}{0\leq n \leq \ell-1}{\PRn[n-\ell](x)}{\ell \leq n,}
	\end{equation}
constitute a family of orthonormal polynomials  with respect to \eqref{Sobolev-IP} on $\WWp[2]$, which is also complete.
\end{theorem}

The algebraic and analytical properties of the polynomials $\{\sn\}_{n\geq0}$  have been studied in \cite{PiRi17}.
\begin{proof} First, note that, for each $n\geq 0$, $\sn$ is a polynomial of degree $n$ and from \eqref{Goncharov_PolyDer} and  \eqref{AG-conditions}
\begin{align}\label{OrtPto}
	\sn^{(k)}\!(\omega_k)=\funD{1}{n=k}{0}{n\neq k,}\qquad \text{for } \;0\leq k\leq \ell-1.
\end{align}
 Now, we are going to prove that $\{\sn\}_{n\geq 0}$ is  an  orthonormal polynomial sequence with respect to \eqref{Sobolev-IP}.

 If $m,n\geq\ell$, then from \eqref{MonicSOP} and \eqref{OrtPto} 
$$
 \IpS{\sn[m]}{\sn}= \int_{-1}^{1}\fpn[m-\ell]{\alpha}{\beta}(x) \,\fpn[n-\ell]{\alpha}{\beta}(x) d\Jm(x)=\funD{1}{n=m}{0}{n\neq m.}
$$
Otherwise, we get that $\sn[m]^{(\ell)}\sn^{(\ell)}\equiv 0$ and from \eqref{OrtPto}
$$
	\IpS{\sn[m]}{\sn}= \sum_{k=0}^{\ell-1} \sn[m]^{(k)}(\omega_{k}) \, \sn^{(k)}(\omega_{k})=\funD{1}{n=m}{0}{n\neq m.}
$$

The completeness is a direct consequence of the  fact that  $\WWp=\overline{\PP}.$
 \end{proof}

Our next step is to define rigorously what means the Fourier series of  $f\in\WWp$, where $p\geq 1$, $\ell\in \NN$, $\alpha,\beta>-1$ and $\varpi=(\omega_0,\omega_1, \dots, \omega_{\ell-1})\in \RR^{\ell}$.

For $f\in\UUp$ we denote
$$ \FSSn(f,x):=   \sum_{k=0}^{n}\IpS{f}{\sn[k]}\; \sn[k](x).$$
Under the above assumptions of  Theorem  \ref{Th-Completeness}, \;   $\WWp={\UUp}$ but, in general, we only know that $\WWp=\overline{\UUp}$. However, if $f\in \WWp\setminus\UUp$, then  $f$ can be identified as a Cauchy sequence $\{f_n\}_{n\geq 0}\subset\UUp$ with norm $\normSp{f}=\lim_{n\to\infty}\normSp{f_n}.$

The existence of this limit is guaranteed since $\{\normSp{f_n}\}_{n\geq 0}$  is a Cauchy sequence in $\RR$. On the other hand, if $\{f_n\}_{n\geq 0}$ is a Cauchy sequence in $\WWp$ we get
\begin{align*}
	\normSp{f_n-f_m}^p=\sum_{k=0}^{\ell-1}\left|f^{(k)}_n(\omega_{k})-f^{(k)}_m(\omega_{k})\right|^p + \normJLp{f^{(\ell)}_n(x)-f_{m}^{(\ell)}}^p,
\end{align*}
where the norm  $\normJLp{\cdot}$ was defined in \eqref{Jacobi-pNorm}. Therefore, the sequences $\{f^{(k)}_n(\omega_k)\}_{n\geq 0}$, $k=0,1,\dots,\ell-1$ and $\{f^{(\ell)}_{n}\}_{n\geq 0}$ are also Cauchy sequences on the complete spaces $\RR$ and $\JLp$, respectively. So there exist the limits
$$
f^{k}_{\omega_k}=\lim_{n\to\infty}f^{(k)}_{n}(\omega_k), \; k=0,1,\dots,\ell-1;  \; \text{ and } \;	f^{\ell}=\lim_{n\to\infty} f^{(\ell)}_n.$$
Then, we get
$$\normSp{f}^p =\lim_{n\to\infty}\normSp{f_n}^p=\lim_{n\to\infty} \sum_{k=0}^{\ell-1}\left|f^{(k)}_n(\omega_{k})\right|^p +  \normJLp{f^{(\ell)}_n}^p=\sum_{k=0}^{\ell-1}\left|f^{k}_{\omega_{k}}\right|^p + \normJLp{f^{\ell}}^p .
$$

 From the proof of Theorem \ref{Th-Completeness}, we see that there is not always a suitable function $f$ satisfying
 \begin{equation*}
  f^{(k)}(\omega_k)=f^{k}_{\omega_{k}}, \text{ for } k=0,1,\cdots,\ell-1; \quad \text{ and } \quad f^{(\ell)}\equiv f^{\ell} \ \text{ a.e. on } [-1,1].
 \end{equation*}
But even so, we can understand the value $f^{k}_{\omega_{k}}$ as the $k$-th derivative of the equivalent class $f$ at $\omega_k$ and the function $f^{\ell}\in \JLp$ as its $\ell$-th derivative a.e. on $[-1,1]$. Then, in a natural way, we can define for every $ f\in\WWp$
$$
   \IpS{f}{\sn}= \sum_{k=0}^{\ell-1} f^{k}_{\omega_{k}} \sn^{(k)}(\omega_{k}) + \int_{-1}^{1} f^{\ell}(x) \,\sn^{(\ell)}(x) d\Jm(x),
$$
and now the Fourier series of $f $ makes sense for all $f\in\WWp.$

 From \eqref{MonicSOP}-\eqref{OrtPto}, it is straightforward to prove that  if $f \in \WWp$, then
\begin{equation}\label{FourCoef-1}
\IpS{f}{\sn[n]}=\funD{f^{(n)}(\omega_n)}{0\leq n \leq \ell-1}{\IPJ{f^{(\ell)}}{\pn[n-\ell]}}{\ell \leq n.}
\end{equation}

\begin{theorem}\label{Th-FourSer-Converg} Let $ \alpha,\beta > -1,$
\begin{equation}\label{New-MnDef}
  M(\alpha,\beta)=2\,\frac{\max\{\alpha,\beta,-\frac{1}{2}\}+1}{\max\{\alpha,\beta,-\frac{1}{2}\}+\frac{3}{2}}\quad  \text{and}\quad  m(\alpha,\beta)=2\, \frac{\max\{\alpha,\beta,-\frac{1}{2}\}+1}{\max\{\alpha,\beta,-\frac{1}{2}\}+\frac{1}{2}}.
\end{equation}
If 	$M(\alpha, \beta) < p < m(\alpha, \beta)$, then, for any  $f\in\WWp$, its Fourier expansion   with respect of the polynomials $\{q_n\}_{n\geq 0}$, converges to $f$ in $\WWp$, i.e.
$$		\lim\limits_{n\rightarrow \infty } \normSp{f-\FSSn(f,\cdot)}=0.$$
\end{theorem}

It is not difficult to verify that the definition of $M(\alpha,\beta)$ and $m(\alpha,\beta)$ in \eqref{New-MnDef} is the same as in  Theorem  \ref{ThPollard} (when $ \alpha,\beta \geq -1/2$) and in \eqref{Muck-Cond} (with $a=\alpha/p$ and $b =\beta/p$).

Let $ \alpha,\beta > -1$,  $\gamma=\max\{\alpha,\beta\}$,
$$
 M(\gamma)=2\,\frac{\max\{\gamma,-\frac{1}{2}\}+1}{\max\{\gamma,-\frac{1}{2}\}+\frac{3}{2}}\quad  \text{and}\quad  m(\gamma)=2\, \frac{\max\{\gamma,-\frac{1}{2}\}+1}{\max\{\gamma,-\frac{1}{2}\}+\frac{1}{2}}.
$$
$(\gamma,p)$ is said to be a convergence pair if $ M(\gamma) < p < m(\gamma)$. In Figure \ref{Fig-1}, the entire shaded region $\Delta$ (light and dark) is the set of all pairs $(\gamma,p)$ of convergence. The dark region $\Delta_0$, at the left, is the set of $(\gamma,p)$ such that  $\WWp$ is a space of suitable functions;  i.e. $\dsty \WWp=\UUp,$ for all  $\varpi\in\RR^\ell$ (see Corollary \ref{CoroFunc}).
\begin{figure}[ht]
  \centering
  \includegraphics[width=.9\textwidth]{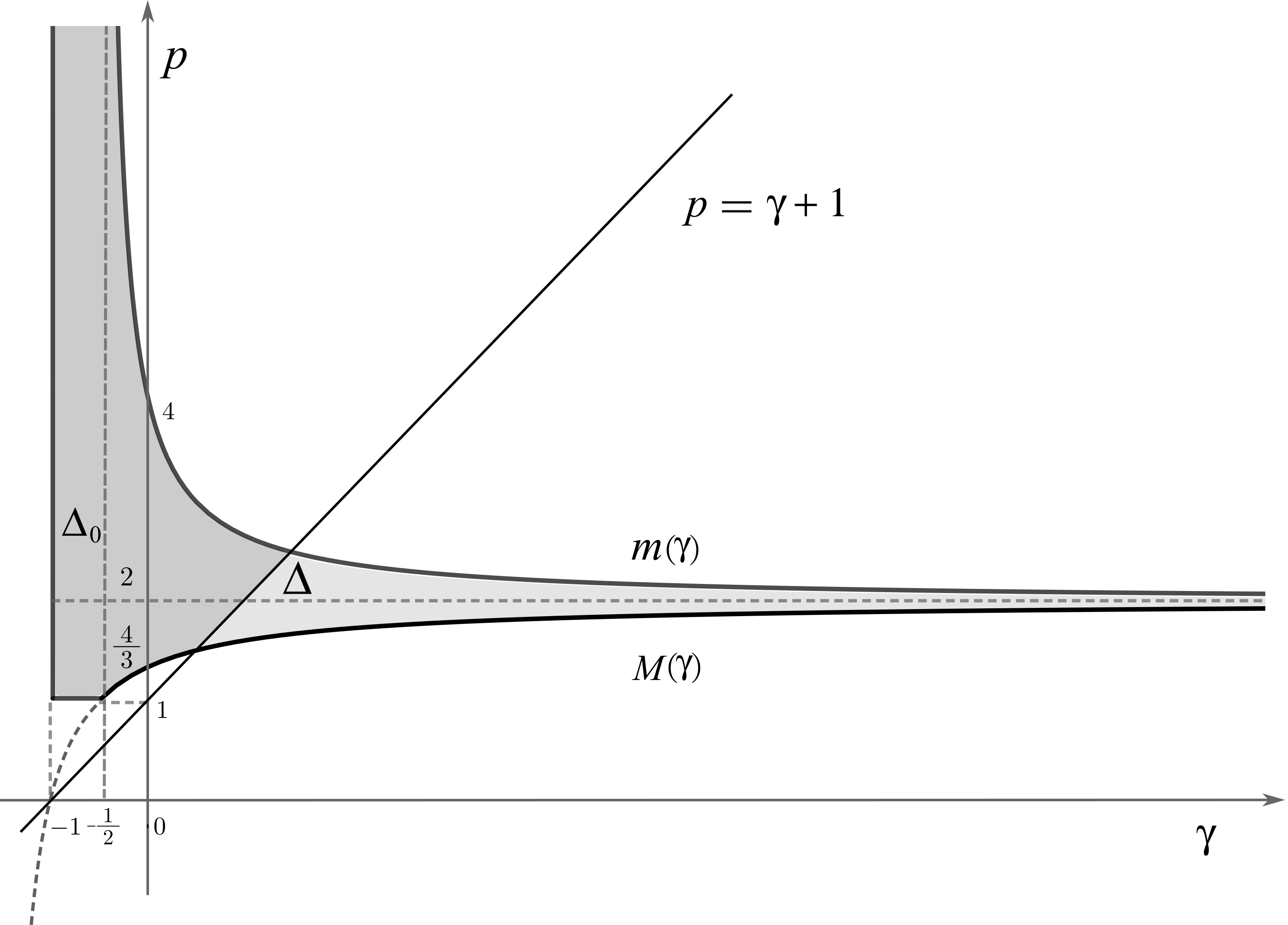}
  \caption{The open regions $\Delta=\{(\gamma,p): M(\gamma) < p < m(\gamma)\}$ and $\Delta_0=\{(\gamma,p)\in \Delta: p>\gamma+1\}$.}\label{Fig-1}
\end{figure}

\begin{proof}[Proof of Theorem \ref{Th-FourSer-Converg}]  Given $f\in\JLp$, let $\FSJn(f,\cdot)$ be the $n$-th partial sum of the Fourier-Jacobi series of $f$
$$\FSJn(f,x)= \sum_{k=0}^{n} \IPJ{f}{\pn[k]}\;  \pn[k](x).$$

Now, given $f  \in\WWp$, from \eqref{FourCoef-1}

1.-  If  $0\leq k \leq \ell-1$, then $\dsty \FSSn^{(k)}\!\!(f,\omega_k)=\funD{0}{n< k}{f^{(k)}(\omega_k)}{n\geq k.}$

2.-  If $n\geq\ell$, then $\FSSn^{(\ell)}\!\!(f,x)=\FSJn[n-\ell](f^{(\ell)}\!,x)$.

From Theorem \ref{ThMuck},  with $a=\alpha/p$ and $b =\beta/p$, there exists a constant $C$, independent on $n$ and $f$, such that
$\normJLp{\FSJn[n-\ell](f^{(\ell)},x)} \leq C \normJLp{f^{(\ell)}}. $

We first prove that there exists a constant $C_1$, independent on $n$ and $f$, such that
\begin{equation}\label{Th5-NormIneq}
  \normSp{\FSSn(f,\cdot)} \leq C_1 \normSp{f}, \qquad \text{for all  $f\in\WWp$.}
\end{equation}
If  $n\geq\ell$, then we have
\begin{align*}
 \normSp{\FSSn(f,\cdot)}^p=&\sum_{k=0}^{\ell-1} \left|\FSSn^{(k)}\!(f,\omega_{k})\right|^{p} + \int_{-1}^{1} \left|\FSSn^{(\ell)}\!(f,x)\right|^{p}\!\!d\Jm(x)\\
 = & \sum_{k=0}^{\ell-1} \left|f^{(k)}(\omega_{k})\right|^{p} + \normJLp{\FSJn[n-\ell](f^{(\ell)},\cdot)}^p \\
\leq&\sum_{k=0}^{\ell-1} \left|f^{(k)}(\omega_{k})\right|^{p} + C^p\normJLp{f^{(\ell)}}^p \leq \max\{1,C^p\}\normSp{f}^p.
\end{align*}
If $0\leq n\leq \ell-1$, then we get
\begin{align*}
 \normSp{\FSSn(f,\cdot)}^p=&\sum_{k=0}^{\ell-1} \left|\FSSn^{(k)}\!(f,\omega_{k})\right|^{p} + \int_{-1}^{1} \left|\FSSn^{(\ell)}\!(f,x)\right|^{p} d\Jm(x)\\
 \leq & \sum_{k=0}^{\ell-1} \left|f^{(k)}(\omega_{k})\right|^{p} \leq \normSp{f}^p.
\end{align*}
This yields formula \eqref{Th5-NormIneq} with $\dsty C_1=\max\{1,C\}$.

Given $f\in\WWp$ and $\epsilon>0$, then there exists $P\in\PP$,  such that $ \normSp{f-P}<\frac{\epsilon}{C_1+1},$ where $C_1$ is given by \eqref{Th5-NormIneq}. Note that for all $n$ greater than or equal to the degree of $P$ we get $P\equiv \FSSn(P,\cdot)$ and, as a consequence,
\begin{align*}
\normSp{\FSSn(f,\cdot)-f}\!\leq &\normSp{\FSSn(f,\cdot)-P}\!+\!\normSp{P-f}\!=\!\normSp{\FSSn(f,\cdot)-\FSSn(P,\cdot)}\!+\!\normSp{P-f}\\
=& \normSp{\FSSn(f-P,\cdot)}\!+\!\normSp{P-f}\leq (C_1+1)\normSp{P-f}< \epsilon.
\end{align*}
Therefore,  $\lim\limits_{n\rightarrow \infty } \normSp{f-\FSSn(f,\cdot)}=0.$
 \end{proof}

\section*{Acknowledgements}
The research of F. Marcell\'{a}n and H. Pijeira-Cabrera was   partially supported by  Ministry of Science, Innovation and Universities of Spain, under grant  PGC2018-096504-B-C33. The research of A. D\'{\i}az-Gonz\'{a}lez was supported by the Research Fellowship Program, Ministry of Economy and Competitiveness of Spain,  under grant  BES-2016-076613.


\end{document}